\documentclass[12pt,oneside]{amsart}
\usepackage{amsmath,amsfonts,amssymb,enumitem,amsthm,epsfig,url,array,tikz,tikz-cd}
\usepackage{mathrsfs,mathtools}
\usetikzlibrary{matrix,arrows,decorations.pathmorphing,positioning,calc}
\usepackage[a4paper, total={5.5in, 9in}]{geometry}

\usepackage[backref]{hyperref}

\theoremstyle{plain}
\newtheorem{theorem}{Theorem}[section]
\newtheorem{corollary}[theorem]{Corollary}
\newtheorem{proposition}[theorem]{Proposition}
\newtheorem{lemma}[theorem]{Lemma}

\theoremstyle{definition}
\newtheorem{example}[theorem]{Example}
\newtheorem{remark}[theorem]{Remark}
\newtheorem{definition}[theorem]{Definition}

\newcommand{\DF}{\mathrm{DF}}

\newcommand{\C}{\mathbb{C}}

\renewcommand{\H}{\mathcal{H}}
\newcommand{\K}{\mathcal{K}}
\newcommand{\M}{\mathcal{M}}
\newcommand{\N}{\mathcal{N}}
\renewcommand{\P}{\mathbb{P}}
\newcommand{\R}{\mathbb{R}}

\newcommand{\X}{\mathcal{X}}
\renewcommand{\L}{\mathcal{L}}
\newcommand{\V}{\mathcal{V}}

\newcommand{\Y}{\mathcal{Y}}
\newcommand{\tY}{\widetilde{\mathcal{Y}}}
\newcommand{\red}{\mathrm{red}}

\newcommand{\Aut}{\mathrm{Aut}}               

\newcommand{\Isom}{\mathrm{Isom}}
\newcommand{\pt}{\mathrm{pt}}

\newcommand{\Ric}{\mathrm{Ric}}
\newcommand{\tr}{\mathrm{tr}}                       

\renewcommand{\d}{\partial}
\newcommand{\db}{\bar{\partial}}

\newcommand{\GL}{\mathrm{GL}}

\renewcommand{\O}{\mathrm{O}}

\subjclass[2020]{53C55, 32Q26}
\begin{document}

\title{Geodesics in the space of relatively K{\"a}hler metrics}

\author{Michael Hallam}
\address{Department of Mathematics, Aarhus University, Ny Munkegade 118, DK-8000 Aarhus C,
Denmark.}
\email{hallam@math.au.dk}

\date{}

\begin{abstract}
	We derive the geodesic equation for relatively K{\"a}hler metrics on fibrations and prove that any two such metrics with fibrewise constant scalar curvature are joined by a unique smooth geodesic. We then show convexity of the log-norm functional for this setting along geodesics, which yields simple proofs of Dervan and Sektnan's uniqueness result for optimal symplectic connections and a boundedness result for the log-norm functional. Next, we associate to a fibration degeneration a unique geodesic ray defined on a dense open subset. Calculating the limiting slope of the log-norm functional along a globally defined smooth geodesic ray, we prove that fibrations admitting optimal symplectic connections are polystable with respect to a large class of fibration degenerations that are smooth over the base. We give examples of such degenerations in the case of projectivised vector bundles and isotrivial fibrations.
\end{abstract}

\maketitle

\section{Introduction}

A highly celebrated result in geometry is the Kobayashi--Hitchin correspondence, which states that a holomorphic vector bundle \(E\) on a compact K{\"a}hler manifold \((X,\omega)\) admits a Hermite--Einstein metric if and only if the bundle is polystable \cite{Don85,UY86}. When the K{\"a}hler class \([\omega]\) is the first Chern class of an ample line bundle \(L\), this reveals a deep link between differential and algebraic geometry. 

Much of modern geometry has centred on finding various analogues of the Kobayashi--Hitchin correspondence in different geometric contexts. One such iteration is the Yau--Tian--Donaldson conjecture \cite{Yau93,Tia97,Don02}, which says that a polarised manifold \((X,L)\) should admit a constant scalar curvature K{\"a}hler (cscK) metric in the ample class \(c_1(L)\) if and only if \((X,L)\) is K-polystable. There are also other examples, such as the J-equation \cite{LS15,Che19} and the deformed Hermitian Yang--Mills equation \cite{CY18,Che19}.

In this paper, we study an appropriate analogue of the Kobayashi--Hitchin correspondence for holomorphic fibrations \(\pi:X\to B\), which was formulated and conjectured in a recent series of works by Dervan--Sektnan \cite{DS_osc,DS_moduli,DS_uniqueness}. On the differential geometric side, the fundamental objects under consideration are relatively K{\"a}hler metrics on \(X\) with fibrewise constant scalar curvature, called \emph{relatively cscK metrics}. Dervan and Sektnan identify a suitable notion of canonical relatively cscK metric, called an \emph{optimal symplectic connection}, together with a notion of polystability for fibrations which they conjecture is equivalent to the existence of an optimal symplectic connection.

Dervan and Sektnan have also proven a partial result in one direction: they show that when the relatively K{\"a}hler class is the first Chern class of a relatively ample line bundle, existence of an optimal symplectic connection implies \emph{semistability} of the fibration \cite{DS_moduli}. The approach taken is an asymptotic analysis of Bergman kernels, analogous to Donaldson's proof that polarised cscK manifolds are K-semistable \cite{Don05}. In this paper we strengthen their result by instead analysing the geometry of the space of relatively cscK potentials \(\K_E\). Not only does this refine their result to \emph{polystability} with respect to a large class of fibration degenerations, but also the proof works in the transcendental K{\"a}hler setting, so extends their result in another direction. 

First, we derive a geodesic equation arising from a natural Riemannian structure on the space \(\K_E\), and then prove: \begin{theorem}
	Any two points of \(\K_E\) are joined by a unique smooth geodesic.
\end{theorem} This is essentially a consequence of the fact that on a fixed K{\"a}hler manifold, any two cscK metrics \(\omega_0\) and \(\omega_1\) are joined by a unique smooth geodesic in the space of K{\"a}hler potentials. This follows from the work of Berman--Berndtsonn \cite{BB17}, indeed there exists an automorphism \(g\) in the identity component of \(\mathrm{Aut}(X)\) such that \(g^*\omega_1=\omega_0\), and \(g\) arises from a holomorphic vector field, which in turn generates a geodesic joining \(\omega_0\) to \(\omega_1\). Using the explicit structure of geodesics, we compute a Riemannian decomposition of \(\K_E\) into the product of a flat manifold and a non-positively curved manifold. This parallels the description of the space of K{\"a}hler potentials as a negatively curved space \cite{Mab87,Sem92,Don99}.

An important tool in the study of canonical metrics is to consider so-called log-norm functionals on spaces of metrics, such as Donaldson's functional on the space of hermitian metrics on a vector bundle \cite{Don85}, or the Mabuchi functional on the space of K{\"a}hler metrics in a fixed K{\"a}hler class \cite{Mab86}. Dervan and Sektnan show that there exists a natural log-norm funtional \(\N:\K_E\to\R\) on the space of relatively cscK potentials, whose critical points are precisely the optimal symplectic connections \cite{DS_uniqueness}. We prove: \begin{theorem}
The functional \(\N\) is convex along smooth geodesics in \(\K_E\).
\end{theorem} This convexity immediately yields a greatly simplified proof of Dervan and Sektnan's recent uniqueness result for optimal symplectic connections, analogous to Donaldson's proof of uniqueness of Hermite--Einstein metrics: \begin{corollary}[{\cite[Theorem 1.1]{DS_uniqueness}}]
Let \(\omega_X\) and \(\omega_X'\) be two optimal symplectic connections in the same cohomology class. Then there exists a holomorphic automorphism \(g:X\to X\) preserving the projection \(\pi:X\to B\), and a smooth function \(f:B\to\R\) such that \[\omega_X'=g^*\omega_X+i\d\db\pi^*f.\]
\end{corollary} Using convexity, we also strengthen another result of Dervan--Sektnan, showing that existence of an optimal symplectic connection implies the functional \(\N\) is bounded.

A fibration is defined to be stable if a certain invariant associated to any nontrivial fibration degeneration \(\X\to B\times\C\) is positive; see Section \ref{sec:fibration_stability} for the precise definitions. The general fibre over \(B\) of a fibration degeneration may be assumed to be a product test configuration; here we restrict attention to a suitable class of ``product-type" fibration degenerations whose fibres are \emph{all} product test configurations. We can then associate to any product-type fibration degeneration a unique compatible geodesic ray in \(\K_E\). The key to this is a reduction of the geodesic equation on a product test configuration to a ``Wess--Zumino--Witten" equation by Donaldson \cite{Don99}. We then calculate the limiting slope of the functional \(\N\) along this geodesic ray, using a Chen--Tian style formula for the functional \(\N\) which we derive, together with results of \cite{BHJ19,Sjo18}. Applying the convexity theorem above, we prove: \begin{theorem}\label{thmC}
	If \(X\to B\) is a fibration admitting an optimal symplectic connection, then it is polystable with respect to product-type fibration degenerations.
\end{theorem} This strengthens the semistability result of Dervan--Sektnan in a large class of cases. In terms stability of vector bundles, restricting attention to product-type fibration degenerations is analogous to considering slope stability only with respect to sub-vector bundles, rather than general coherent subsheaves. In particular, we expect that this should provide an essentially complete proof of polystability of fibrations over curves. We give a large class of examples of product-type fibration degenerations for both projectivised vector bundles and isotrivial fibrations. In particular, we also prove: \begin{corollary}
If \(X\to B\) is an isotrivial fibration admitting an optimal symplectic connection, then it is polystable with respect to fibration degenerations that are smooth over \(B\times\C\).
\end{corollary}

Here smoothness is meant in the sense of scheme theory, so that \(\X\to B\times\C\) is a holomorphic surjective submersion.

The general framework of ideas above has proven tremendously useful in the theory of cscK metrics and K-stability. For instance, the proof of Theorem \ref{thmC} above is analogous to \cite{BHJ19} and \cite{BDL16} in the K{\"a}hler setting. One might hope that these methods will eventually lead to a converse, yielding a full proof of the conjecture. Interestingly, these ideas have also fed back into the theory of vector bundles, particularly in \cite{HK18} where the authors reprove one direction of the Kobayashi--Hitchin correspondence using convexity of Donaldson's functional and non-Archimedean limits.

\subsection*{Outline}

We begin in Section \ref{background} by first reviewing relevant material on K-stability for varieties and K{\"a}hler metrics of constant scalar curvature. We then cover more recent material on stability of fibrations and optimal symplectic connections needed in later sections.

In Section \ref{sec:geodesic_equation} we derive the geodesic equation for relatively cscK potentials, and prove existence and uniqueness of smooth geodesics. We then compute a Riemannian decomposition of the space of relatively cscK potentials, into the product of a flat manifold and a non-positively curved manifold.

Next, in Section \ref{functional} we analyse the log-norm functional along geodesics. In particular, we show that the log-norm functional is convex along smooth geodesics. We then use this to give a very simple proof of uniqueness of optimal symplectic connections, and to prove the log-norm functional is bounded below when an optimal symplectic connection exists.

In Section \ref{sec:Chen--Tian} we prove an analogue of the Chen--Tian formula for the log-norm functional \(\N\), which is achieved by expanding the usual Chen--Tian formula for the Mabuchi functional in an adiabatic limit.

Finally in Section \ref{polystability} we show that fibrations admitting optimal symplectic connections are polystable with respect to product-type fibration degenerations. In particular, we associate to any such fibration degeneration a unique smooth geodesic ray. Calculating the limiting slope of the log-norm functional along this ray using the Chen--Tian style formula from Section \ref{sec:Chen--Tian}, we recover the invariant computing stability of the fibration. We end by giving a large class of examples of product-type fibration degenerations for projectivised vector bundles and also for isotrivial fibrations.

\renewcommand{\abstractname}{Acknowledgements}
\begin{abstract}
I owe many thanks to my supervisors: to Ruadha{\'i} Dervan for recommending this problem and patiently guiding me through it, and to Frances Kirwan for plenty of helpful discussions. I also thank John McCarthy for his help with Example \ref{ex:principal_bundle_fibration}, which appears in Chapter 7 of his PhD thesis \cite{McC22}. I lastly thank the referee, for many helpful comments and suggestions. This research was conducted at Oxford University under the support of a Mathematical Institute Studentship.
\end{abstract}

\section{Background}\label{background}

In this section we outline the relevant background material on the differential geometry and algebro-geometric stability of K{\"a}hler manifolds and fibrations. An excellent reference for the setting of K{\"a}hler manifolds is the book \cite{Sze14}, and the notions relevant for fibrations may be found in the papers by Dervan and Sektnan \cite{DS_osc,DS_moduli,DS_uniqueness}.

\subsection{K{\"a}hler geometry}\label{sec:Kahler_background}

Let \((Y,\omega)\) be a compact K{\"a}hler manifold of dimension \(d\). The \emph{scalar curvature} of \(\omega\) is the smooth function \(S(\omega)\) on \(Y\) given by contracting the Ricci cuvature of \(\omega\), \[\Ric(\omega):=-\frac{i}{2\pi}\d\db\log\omega^d,\] with the metric \(\omega\) itself: \[S(\omega):=\Lambda_\omega \Ric(\omega).\] We will be particularly interested in K{\"a}hler metrics \(\omega\) whose scalar curvature \(S(\omega)\) is a constant function; such metrics will be abbreviated \emph{cscK} (constant scalar curvature K{\"a}hler). The average of \(S(\omega)\) is a topological constant, given by \[\hat{S}=\frac{d\,c_1(Y)\cdot[\omega]^{d-1}}{[\omega]^d}\] where \([\omega]\) is the K{\"a}hler class of \(\omega\). Thus, the metric \(\omega\) is cscK if and only if \(S(\omega)=\hat{S}\). For links to algebraic geometry, we will later consider the case where \([\omega]\) is equal to the first Chern class \(c_1(L)\) of an ample line bundle \(L\to Y\).

 For a fixed K{\"a}hler metric \(\omega\), let \[\K:=\{\phi\in C^\infty(Y,\R):\omega+i\d\db\phi>0\}\] be the space of K{\"a}hler potentials with respect to \(\omega\). We will use the notation \(\omega_\phi:=\omega+i\d\db\phi\) whenever \(\phi\in \K\). Any K{\"a}hler metric in the K{\"a}hler class \([\omega]\) may be written as \(\omega_\phi\) for some \(\phi\in\K\) that is unique up to addition of a real constant.  Since \(\K\) is an open subset of the Fr{\'e}chet space \(C^\infty(Y,\R)\), it is naturally a Fr{\'e}chet manifold with tangent space isomorphic to \(C^\infty(Y,\R)\) at each point.

There is a natural \(L^2\)-inner product on each tangent space \(T_\phi\K\cong C^\infty(Y,\mathbb{R})\) defined by \[\langle\psi_1,\psi_2\rangle_\phi:=\int_Y\psi_1\psi_2\,\omega_\phi^d.\] This in turn gives rise to an energy functional \(\mathcal{P}\) on smooth paths \(\phi:[0,1]\to\mathcal{K}\), \[\mathcal{P}(\phi):=\int_0^1\langle\dot{\phi}_t,\dot{\phi}_t\rangle_{\phi_t}\,dt.\] A smooth path \(\phi:[0,1]\to\mathcal{K}\) is a \emph{geodesic} if it is a critical point of the energy functional \(\mathcal{P}\). Taking the variation of the functional \(\mathcal{P}\), one shows that a smooth path \(\phi:[0,1]\to\K\) is a geodesic if and only if it satisfies the PDE \[\ddot{\phi}_t=|\d\dot{\phi}_t|^2_{t}\] for all \(t\in[0,1]\), where \(|\cdot|_t\) is the hermitian metric on \(\Lambda^{1,0}T^*Y\) determined by \(\omega_t:=\omega+i\d\db\phi_t\); we refer to \cite[Proposition 4.25]{Sze14} for a detailed proof. 
	
It is an important fact that any two points of \(\K\) are joined by a unique geodesic \cite{Che00}, although the geodesic may not be smooth and one must enlarge the space \(\K\) to include K{\"a}hler potentials that are only \(C^{1,1}\)-regular \cite{CTW18}. 

A very useful class of geodesics can be constructed through holomorphic vector fields on \(Y\), which are in turn described by holomorphy potentials. A \emph{holomorphy potential} (with respect to \(\omega\)) is a smooth function \(u:Y\to\C\) such that \(\nabla^{1,0}u\) is a holomorphic vector field, where \(\nabla\) is the gradient operator determined by \(\omega\) and \(\nabla^{1,0}u:=(\nabla u)^{1,0}\). Equivalently, a holomorphy potential is a function in the kernel of the operator \(\mathcal{D}:C^\infty(Y,\C)\to\Omega^{0,1}(\Lambda^{1,0}TY)\) defined by \[\mathcal{D}(u)=\db\nabla^{1,0}u.\] The operator \(\mathcal{D}^*\mathcal{D}:C^\infty(Y,\C)\to C^\infty(Y,\C)\) is called the \emph{Lichnerowicz operator}; it is a fourth order elliptic operator, and hence its kernel is finite dimensional. 

Define \[E:=\{f:Y\to\mathbb{R}:\mathcal{D}(f)=0\text{ and } \int_Y f\,\omega^d=0\},\] and let \(\mathfrak{g}\) be the finite dimensional vector space of holomorphic vector fields on \(Y\) that are equal to \(\nabla^{1,0}u\) for some holomorphy potential \(u\). The operator \(\nabla^{1,0}:E\oplus iE\to\mathfrak{g}\) is then an isomorphism. We remark that \(\mathfrak{g}\) is exactly the space of holomorphic vector fields which vanish somewhere \cite[Theorem 1]{LS94}, so is independent of the metric \(\omega\). It is also a Lie algebra under the Lie bracket of vector fields.

We denote by \(\Aut(Y)\) the complex Lie group of biholomorphisms of \(Y\), and by \(\Aut_0(Y)\) its identity component. Given a holomorphic vector field \(v\in H^0(Y,TY)\), we let \(v_{\mathbb{R}}:=(v+\bar{v})/2\) be its real part, and define \(\exp(tv)\) to be the time \(t\) flow of the real holomorphic vector field \(v_{\mathbb{R}}\), which is a biholomorphism of \(Y\). This notation comes from Lie group theory: here \(v\) is an element of the Lie algebra \(H^0(Y,TY)\) of holomorphic vector fields, and \(\exp(tv)\in\Aut_0(Y)\) is the image of \(tv\) under the exponential map of Lie groups. If \(v=\nabla^{1,0}u\) for some \(u\in E\), then \(v_{\mathbb{R}}=\frac{1}{2}\nabla u\).

Our interest in holomorphy potentials comes from the following result:

\begin{proposition}\label{prop:holomorphy_geodesics}
	Let \(\omega\) be a cscK metric, let \(u\in E\), and define \(v:=\nabla^{1,0}u\). For \(t\in\mathbb{R}\), denote by \(\exp(tv)\) the time \(t\) flow of \(v_{\mathbb{R}}=\frac{1}{2}\nabla u\). Then \begin{equation}\label{eq:holomorpy_geodesic}
	\phi_t:=\int_0^t\exp(sv)^*(u)\,ds
	\end{equation} defines a smooth geodesic in \(\mathcal{K}\). Writing \(\omega_t:=\omega+i\d\db\phi_t\), we have \begin{equation}\label{eq:pullback_formula}
	\omega_t=\exp(tv)^*\omega
	\end{equation} for all \(t\), and hence \(\omega_t\) is a cscK metric. Furthermore, the geodesic satisfies \begin{equation}\label{eq:integral_zero}
	\int_Y\dot{\phi}_t\,\omega_t^d=0
	\end{equation} for all \(t\), and if \(u\neq0\) then the geodesic is nontrivial.
	
	Conversely, suppose that \(\omega\) and \(\omega'\) are two cscK metrics in the same K{\"a}hler class, and let \(E\) be the space of real integral-zero holomorphy potentials with respect to \(\omega\). Then there exists a unique element \(u\in E\) such that the geodesic \(\phi_t\) defined by \eqref{eq:holomorpy_geodesic} satisfies \(\omega+i\d\db\phi_1=\omega'\). Thus, the set of cscK metrics in \([\omega]\) is in bijection with \(E\).
\end{proposition}

The forward implication (that \(\phi_t\) defines a geodesic satisfying \eqref{eq:pullback_formula} and \eqref{eq:integral_zero}) is well known (see \cite[Example 4.26]{Sze14}, for example), although we will give a detailed proof for completeness. The reverse implication is a simple consequence of the deep uniqueness theorem of Berman--Berndtsson for cscK metrics \cite[Theorem 1.3]{BB17}.

Before proving the theorem, we must better understand the space of cscK metrics. For a K{\"a}hler metric \(\omega\), we denote by \(\Isom(\omega)\) the real Lie group of isometries of \((Y,\omega)\), and by \(\Isom_0(\omega)\) its identity component. Note that these are real compact Lie subgroups of \(\Aut(Y)\) and \(\Aut_0(Y)\), respectively. The Lie algebra \(\mathfrak{g}\) of holomorphic vector fields arising from holomorphy potentials is the Lie algebra of a closed Lie subgroup \(\Aut_{\red}(Y)\) of \(\Aut_0(Y)\) called the \emph{reduced automorphism group}. Concretely, \(\Aut_{\red}(Y)\) is the kernel of the map \(\Aut_0(Y)\to\mathrm{Alb}(Y)\), where \(\mathrm{Alb}(Y)\) is the Albanese torus of \(Y\) \cite[Section 2.4]{Gau10}. 

In the case \(\omega\) is a cscK metric, we have the following:

\begin{theorem}[Partial Cartan decomposition / Matsushima--Lichnerowicz Theorem]\label{thm:decomposition}
	Let \((Y,\omega)\) be a compact K{\"a}hler manifold with constant scalar curvature. Denote by \(\mathfrak{g}\) the complex Lie algebra of holomorphic vector fields generated by holomorphy potentials, and let \(\mathfrak{k}\) denote the real Lie subalgebra of \(\mathfrak{g}\) consisting of Killing vector fields for \(\omega\) in \(\mathfrak{g}\). Then: \begin{enumerate} 
		\item there is a splitting \[\mathfrak{g}=\mathfrak{k}\oplus i\mathfrak{k}\] of real vector spaces, and
		\item \(\mathfrak{k}\) is generated by the \emph{purely imaginary} holomorphy potentials.
	\end{enumerate} 
	Furthermore, let \(G:=\Aut_{\red}(Y)\) and \(K:=\mathrm{Isom}_0(\omega)\cap G\). Then: \begin{enumerate}\setcounter{enumi}{2}
		\item \(\mathfrak{g}\) is the Lie algebra of \(G\), and
		\item \(K\subset G\) is a maximal compact subgroup whose Lie algebra is \(\mathfrak{k}\).
	\end{enumerate}  
	Lastly, write \(\tilde{\mathfrak{g}}\) for the Lie algebra \(H^0(Y,TY)\) of all holomorphic vector fields, and let \(\mathfrak{a}\subset\tilde{\mathfrak{g}}\) be the abelian ideal of harmonic vector fields. Write \(\tilde{G}:=\Aut_0(Y)\) and \(\tilde{K}:=\Isom_0(\omega)\). Then: \begin{enumerate}\setcounter{enumi}{4}
	\item the Lie algebra \(\tilde{\mathfrak{k}}\) of \(\tilde{K}\) satisfies: \[\tilde{\mathfrak{k}}=\mathfrak{a}\oplus\mathfrak{k},\] and 
	\item the map \(\tilde{K}\times i\mathfrak{k}\to\tilde{G}\), \((k,v)\mapsto k\exp(v)\) is surjective.
	\end{enumerate}
\end{theorem} 

We refer to \cite[Proposition 2.4.2 and Sections 3.4--3.6]{Gau10} and \cite[Section 6.1]{DaR17} for proofs of these facts. In Lie group terms, the above theorem says the Lie group \(G\) is a complex reductive group, and is the complexification of the maximal compact subgroup \(K\). In this situation, we may apply the following lemma:

\begin{lemma}\label{lem:homog_space}
Let \(G\) be a connected complex Lie group and let \(K\subset G\) be a maximal compact subgroup. Denote by \(\mathfrak{g}\) and \(\mathfrak{k}\) the Lie algebras of \(G\) and \(K\), respectively. If \(\mathfrak{g}=\mathfrak{k}\oplus i\mathfrak{k}\), then the map \[i\mathfrak{k}\to G/K,\quad\quad v\mapsto K\exp(v)\] is a diffeomorphism, where \(G/K\) is the space of right cosets of \(K\), and \(\exp:\mathfrak{g}\to G\) is the exponential map of Lie groups. Furthermore, the space \(G/K\) is non-positively curved, and the geodesics in \(G/K\) emanating from \(K\cdot 1\) all take the form \(t\mapsto K\exp(tv)\) for some \(v\in i\mathfrak{k}\).
\end{lemma}

See \cite[Appendix C]{GRS21} for an account of this fact. The significance of the homogeneous space \(G/K\) is the following:

\begin{lemma}\label{lem:cscK_space}
	Let \(\omega\) be a cscK metric on \(Y\). As above, denote by \(G\) the reduced automorphism group of \(Y\), and \(K\) the subgroup of \(G\) consisting of isometries of \(\omega\) in \(G\). Then the homogeneous space \(G/K\) of right cosets of \(K\) is in bijection with the set of cscK metrics in the K{\"a}hler class \([\omega]\).
\end{lemma}

\begin{proof}
	If \(\omega'\) is a cscK metric in the K{\"a}hler class \([\omega]\), then by \cite[Theorem 1.3]{BB17} there exists a biholomorphism \(g\in\tilde{G}:=\Aut_0(Y)\) so that \(\omega'=g^*\omega\). By \emph{(6)} of Theorem \ref{thm:decomposition}, there exist elements \(k\in\tilde{K}:=\Isom_0(\omega)\) and \(v\in i\mathfrak{k}\) such that \(g=k\exp(v)\). Since \(\mathfrak{k}\oplus i\mathfrak{k}\) is the Lie algebra of \(G\), the element \(\exp(v)\) lies in \(G\), and furthermore satisfies \(\exp(v)^*\omega=\exp(v)^*k^*\omega=g^*\omega=\omega'\). Hence we may take \(g\) to lie in \(G\) rather than \(\tilde{G}\). Since \(K\) consists of the isometries of \(\omega\) in \(G\), we can associate to \(\omega'\) the unique coset \(Kg\). Conversely, for any \(g\in G\) the metric \(g^*\omega\) is cscK, is unchanged by left-multiplying \(g\) by an element of \(K\), and is in the class \([\omega]\). Thus, we can also associate to any element of \(G/K\) a unique cscK metric.
\end{proof}

We now have enough information to prove Proposition \ref{prop:holomorphy_geodesics}.

\begin{proof}[Proof of Proposition \ref{prop:holomorphy_geodesics}] First, differentiating \eqref{eq:holomorpy_geodesic}, \begin{equation}\label{eq:first_derivative}
\dot{\phi}_t=\exp(tv)^*(u).
\end{equation} It follows that \[\frac{d}{dt}(\omega+i\d\db\phi_t)=\exp(tv)^*(i\d\db u).\] By definition of the gradient, \(\omega(v_{\R},-)=d^cu\), where \(d^c:=\frac{i}{2}(\db-\d)\). Hence \begin{align*}
\frac{d}{dt}(\exp(tv)^*\omega) &= \exp(tv)^*(\mathcal{L}_{v_{\R}}\omega) \\
&= \exp(tv)^*(d(\omega(v_{\R},-))) \\
&= \exp(tv)^*(dd^cu) \\
&= \exp(tv)^*(i\d\db u),
\end{align*} where \(\mathcal{L}_{v_{\R}}\) is the Lie derivative with respect to \(v_{\R}\) and we have used Cartan's magic formula in the second line. Since \(\omega+i\d\db\phi_0=\omega=\exp(0v)^*\omega\), the equality \(\exp(tv)^*\omega=\omega+i\d\db\phi_t\) holds for all \(t\), proving \eqref{eq:pullback_formula}. 

We easily see \eqref{eq:integral_zero} by \[\int_Y\dot{\phi}_t\,\omega_t^d=\int_Y\exp(tv)^*(u)\,(\exp(tv)^*\omega)^d=\int_Yu\,\omega^d=0,\] where we have used \eqref{eq:first_derivative}, \eqref{eq:pullback_formula}, that \(\exp(tv)\) is a diffeomorphism, and that \(u\in E\) means \(u\) has integral zero with respect to \(\omega\).

Denote by \(g_t\) the Riemannanian metric corresponding to the K{\"a}hler metric \(\omega_t\), and let \(g:=g_0\). Differentiating \eqref{eq:first_derivative}, \begin{align*}
\ddot{\phi}_t &= \exp(tv)^*(\mathcal{L}_{v_{\R}}(u)) \\
&= \exp(tv)^*\left(\frac{1}{2}g(\nabla u,\nabla u)\right) \\
&= \exp(tv)^*(|\d u|^2).
\end{align*} Since \(\exp(tv)\) is a biholomorphism, and since \(\exp(tv)^*\omega=\omega_t\), we have \(\exp(tv)^*g=g_t\). Combining this with \eqref{eq:first_derivative}, \[\exp(tv)^*(|\d u|^2)=|\d(\exp(tv)^*(u))|_t^2=|\d\dot{\phi}_t|_t^2,\] and so the geodesic equation \(\ddot{\phi}_t=|\d\dot{\phi}_t|_t^2\) holds. If \(u\neq0\) then \(\dot{\phi}_0\neq0\), and the geodesic is nontrivial. This completes the proof of the forward implication.

For the converse, let \(\omega\) and \(\omega'\) be cscK metrics in the same K{\"a}hler class. By Lemma \ref{lem:cscK_space}, the metric \(\omega'\) corresponds to an element \(Kg\) in the homogeneous space \(G/K\). By Lemma \ref{lem:homog_space}, \(Kg\) is equal to \(K\exp(v)\), for a unique element \(v\in i\mathfrak{k}\). By \emph{(2)} of Theorem \ref{thm:decomposition}, \(i\mathfrak{k}\) is isomorphic to \(E\) via \(\nabla^{1,0}\), so there exists a unique element \(u\in E\) such that \(\omega'\) corresponds to \(G\exp(\nabla^{1,0}u)\) under the isomorphism of Lemma \ref{lem:cscK_space}. In particular, \(\omega'=\exp(\nabla^{1,0}u)^*\omega\). We now see from \eqref{eq:pullback_formula} that the geodesic \(\phi_t\) defined by \eqref{eq:holomorpy_geodesic} satisfies \(\omega'=\exp(\nabla^{1,0}u)^*\omega=\omega+i\d\db\phi_1\).
\end{proof}

The existence of a cscK metric in a fixed K{\"a}hler class is closely related to the properties of the following functional on the space of K{\"a}hler potentials: fixing a K{\"a}hler metric \(\omega\) on \(Y\), there exists a unique functional \(\mathcal{M}:\K\to\R\) such that \(\mathcal{M}(0)=0\) and \begin{equation}\label{eq:Mabuchi}
	\frac{d}{dt}\mathcal{M}(\phi_t)=\int_Y\dot{\phi}_t(\hat{S}-S(\omega_t))\,\omega_t^d,
	\end{equation} for any smooth path \(\phi:[0,1]\to\K\) \cite{Mab86}. The functional \(\M\) is called the \emph{Mabuchi functional}; it depends on the choice of fixed K{\"a}hler metric \(\omega\), but in an inessential way. Note that critical points of the Mabuchi functional correspond to cscK metrics in the class \([\omega]\).

\subsection{K-stability} Consider a pair \((Y,L)\), where \(Y\) is a complex normal projective variety of dimension \(d\) and \(L\) is an ample line bundle on \(Y\). Such a pair \((Y,L)\) is called a \emph{polarisied variety}. It is known that when \(Y\) is smooth and admits a cscK metric in the K{\"a}hler class \(c_1(L)\), the polarised variety \((Y,L)\) is K-polystable \cite{BDL16}. Here, K-polystability is an algebro-geometric stability condition, defined by analogy with the Hilbert-Mumford condition from Geometric Invariant Theory, which we will now describe. 

\begin{definition}
	Let \((Y,L)\) be a polarised variety. A \emph{test configuration} \((\Y,\L)\) for \((Y,L)\) consists of: \begin{enumerate}[label=(\roman*)]
		\item a variety \(\Y\) with a \(\C^*\)-equivariant flat proper morphism \(\Y\to\C\), and
		\item a \(\C^*\)-equivariant relatively ample line bundle \(\L\to\Y\),
	\end{enumerate} such that there is an isomorphism \((\Y_1,\L_1)\cong(Y,L^r)\) for some \(r>0\) called the \emph{exponent} of the test configuration. 
\end{definition}

Two test configurations \((\Y,\L)\) and \((\Y',\L')\) are \emph{isomorphic} if there is a \(\C^*\)-equivariant isomorphism \((\Y,\L)\xrightarrow{\sim}(\Y',\L')\) preserving the projections to \(\C\).

\begin{example}
	Given a one-parameter subgroup \(\rho:\C^*\to\Aut(Y,L)\) we define a test configuration \((\Y,\L)_\rho\) by taking \(\Y:=Y\times\C\) with the diagonal \(\C^*\)-action and \(\L:=p^*L\) with the induced action from \(\rho\), where \(p:Y\times\C\to Y\) is the projection. Any test configuration isomorphic to one of the form \((\Y,\L)_\rho\) for some \(\rho:\C^*\to\Aut(Y,L)\) is called a \emph{product test configuration}. If \(\rho\) is trivial, we call the test configuration \emph{trivial}.
\end{example}

Given a test configuration \((\Y,\L)\), denote by \(A_k\) the infinitesimal generator of the \(\C^*\)-action on \( H^0(\Y_0,\L_0^k)\), and let \(w_k:=\tr A_k\). Then  there are expansions \begin{align*}
\dim H^0(\Y_0,\L_0^k)&=a_0k^d+a_1k^{d-1}+\O(k^{d-2}), \\ w_k&=b_0k^{d+1}+b_1k^d+\O(k^{d-1}),
\end{align*} which hold for all \(k\gg0\). The \emph{Donaldson--Futaki} invariant of \((\Y,\L)\) is defined as \[\DF(\Y,\L):=\frac{a_1b_0-a_0b_1}{a_0^2}.\] 

Recall that a scheme \(Y\) over the complex numbers is called \emph{normal} if its local ring at each point is an integrally closed domain; in particular the singular locus of \(Y\) has codimension at least 2. Any reduced scheme \(Y\) has a unique \emph{normalisation}, that is a finite birational morphism \(\widetilde{Y}\to Y\) from a normal scheme \(\widetilde{Y}\) to \(Y\). If \(Y\) is normal and \((\Y,\L)\) is a test configuration, \(\tY\) is the normalisation of \(\Y\) and \(\widetilde{\L}\) is the pullback of \(\L\) to \(\tY\), then \((\tY,\widetilde{\L})\) is also a test configuration for \((Y,L)\) called the \emph{normalisation} of \((\Y,\L)\). A test configuration \((\Y,\L)\) is called \emph{normal} if its total space \(\Y\) is a normal scheme.

\begin{definition}[{\cite{Tia97,Don02}}]
	A polarised variety \((Y,L)\) is: \begin{enumerate}[label=(\roman*)]
		\item \emph{K-semistable} if \(\DF(\Y,\L)\geq0\) for all test configurations \((\Y,\L)\) for \((Y,L)\),
		\item \emph{K-polystable} if it is K-semistable and \(\DF(\Y,\L)=0\) only if \((\Y,\L)\) normalises to a product test configuration for \((Y,L)\),
		\item \emph{K-stable} if it is K-semistable and \(\DF(\Y,\L)=0\) only if \((\Y,\L)\) normalises to the trivial test configuration for \((Y,L)\),
		\item \emph{K-unstable} if it is not K-semistable.
	\end{enumerate} 
\end{definition}

\begin{remark}\label{rem:tc_normal}
	By a result of Ross--Thomas \cite{RT07}, the normalisation \((\tY,\widetilde{\L})\) of a test configuration \((\Y,\L)\) satisfies \(\DF(\widetilde{\Y},\widetilde{\L})\leq\DF(\Y,\L)\). Hence it suffices to consider only normal test configurations when testing K-(semi/poly)stability.
\end{remark}

When \((\Y,\L)\) is a normal test configuration, there is also an important intersection theoretic formula for the Donaldson--Futaki invariant. Since \((\Y,\L)|_{\C^*}\cong(Y\times\C^*,L)\) via the \(\C^*\)-action, the test configuration \((\Y,\L)\) admits a trivial compactification over \(\infty\in\P^1\). We denote this compactification by \((\overline{\Y},\overline{\L})\). There is also a relative canonical bundle \(K_{\overline{\Y}/\P^1}\) on \(\overline{\Y}\), which is only a Weil divisor in general. By {\cite{Wan12,Oda13}}, if \((\Y,\L)\) is a normal test configuration for \((Y,L)\)  then \[\DF(\Y,\L)=\frac{1}{d+1}\mu(Y,L)(\overline{\L}^{\,d+1})+(\overline{\L}^{\,d}\cdot K_{\overline{\Y}/\P^1}),\] where \(\mu(Y,L):=\frac{-dK_Y\cdot L^{d-1}}{L^d}\).

\subsection{Differential geometry of fibrations}\label{fib_diff_geom}

In this section we will consider holomorphic surjective submersions \(\pi:X\to B\), where \(X\) and \(B\) are compact K{\"a}hler manifolds. By Ehresmann's theorem such a morphism \(\pi\) is then (smoothly) a locally trivial fibration; all the fibres \(X_b\) for \(b\in B\) are diffeomorphic, although the complex structures may vary. We write \(n\) for the dimension of \(B\) and \(m\) for the dimension of each fibre \(X_b\). We fix a K{\"a}hler class \([\omega_B]\) on \(B\) and a relatively K{\"a}hler class \([\omega_X]\) for \(\pi:X\to B\). For each \(b\in B\), we will denote by \(\omega_b\) the K{\"a}hler metric \(\omega_X|_{X_b}\) on the fibre \(X_b\).

To study the differential geometry of such fibrations, we will further make the following assumptions: \begin{enumerate}
	\item Each fibre \(X_b\) admits a cscK metric in the class \([\omega_b]=[\omega_X]|_{X_b}\). 
	\item  \(\dim_\R E_b\) is independent of \(b\), where \(E_b\) is the space of real holomorphy potentials on the fibre \(X_b\) having zero integral with respect to \(\omega_b\).
\end{enumerate}

\begin{example}
	\begin{enumerate}[label=(\roman*)]
		\item Let \(V\to B\) be a holomorphic vector bundle, let \(X:=\P(V)\) be its projectivisation, and let \(H:=\mathcal{O}_{\P(V)}(1)\) be the relative \(\mathcal{O}(1)\)-line bundle.  Then \((X,c_1(H))\) satisfies the hypotheses above.
		\item More generally, let \(P\to B\) be a holomorphic principal \(G\)-bundle, where \(G\) is a reductive group, and let \((Y,L_Y)\) be a polarised manifold admitting a cscK metric in \(c_1(L_Y)\). Given a representation \(\rho:G\to\Aut(Y,L_Y)\), we can form the associated fibre bundle \(X:=P\times_\rho Y\), whose fibre over any point of \(b\) is \(Y\). There is also an induced line bundle \(H:=P\times_\rho L_Y\) whose restriction to any fibre \(Y\) is \(L_Y\). Hence \((X,c_1(H))\) satisfies the above hypotheses. If \(G=\GL(N)\) and \((Y,L_Y)=(\P^{N-1},\mathcal{O}(1))\) then we recover the previous example.
	\end{enumerate}
\end{example}

\begin{definition}
	A \emph{relatively cscK metric} is a relatively K{\"a}hler metric \(\omega_X\) whose restriction \(\omega_b\) to each fibre \(X_b\) has constant scalar curvature.
\end{definition}

If every class \([\omega_b]\) admits a cscK metric, it is then a non-trivial fact that there exists a relatively cscK metric in \([\omega_X]\); see \cite[Lemma 3.8]{DS_osc}, which uses the deformation theory for cscK metrics of \cite{Sze10,Bro11}. Thus, the first assumption is equivalent to assuming the existence of a relatively cscK metric in the relatively K{\"a}hler class.

 By the partial Cartan decomposition in Theorem \ref{thm:decomposition}, the second assumption is equivalent to assuming that the space of holomorphic vector fields \(H^0(X_b,TX_b)\) on the fibre \(X_b\) has complex dimension independent of \(b\in B\).

\begin{proposition}[{\cite[p.~13]{DS_osc}}]\label{prop:bundle_E}
	If \(\dim_{\mathbb{R}}E_b\) is independent of \(b\), then there exists a smooth vector bundle \(E\to B\) whose fibre over \(b\in B\) is \(E_b\), such that a smooth section of \(E\) corresponds precisely to a smooth function on \(X\) restricting to a real holomorphy potential on each fibre with zero integral.
\end{proposition}

The \(L^2\)-inner product on smooth real-valued functions on \(X\) is given by \begin{equation}\label{eq:inner_product}
\langle f,g\rangle:=\int_X fg\,\omega_X^m\wedge\omega_B^n,
\end{equation} where we abusively write \(\omega_B\) for what is really \(\pi^*\omega_B\). There is then an orthogonal direct sum decomposition: \begin{equation}\label{eq:orthog_decomp}
C^\infty(X)= C^\infty(B)\oplus C^\infty(E)\oplus C^\infty_R(X),
\end{equation} where \(C^\infty(B)\) is the space of smooth functions on \(B\) pulled back to \(X\), \(C^\infty(E)\) denotes the smooth sections of \(E\to B\), and \(C^\infty_R(X)\) is the orthogonal complement of \(C^\infty(B)\oplus C^\infty(E)\) in \(C^\infty(X)\).

\begin{definition}
	Given a fixed relatively cscK metric \(\omega_X\), define \(\mathcal{K}_E\subset C^\infty(X,\R)\) to be the space of smooth functions \(\varphi:X\to\R\) such that \(\omega_X+i\d\db\varphi\) is also relatively cscK.
\end{definition}

The notation \(\K_E\) comes from the following computation of its tangent space:

\begin{proposition}[{\cite[Lemma 4.20]{DS_uniqueness}}]\label{prop:K_E_tangent_decomp}
	Let \(\varphi_t:[0,1]\to\K_E\) be a smooth path. If \(E_t\) denotes the bundle of fibrewise holomorphy potentials corresponding to \(\omega_t:=\omega_X+i\d\db\varphi_t\), then \(\dot{\varphi}_t\in C^\infty(B)\oplus C^\infty(E_t)\) for all \(t\).
\end{proposition}

We often think of functions in \(C^\infty(B)\subset C^\infty(X)\) as in some sense irrelevant, which is partly justified by Lemma \ref{lem:base_form} and Theorem \ref{thm:OSC_uniqueness} below. We will see in Section \ref{sec:geodesic_equation} that the smooth functions on \(B\) play a role analogous to that of the constant functions in the usual theory of K{\"a}hler metrics and potentials.

\begin{lemma}[{\cite[Lemma 3.9]{DS_osc}}]\label{lem:base_form}
	Let \(\eta\) and \(\rho\) be two smooth real closed \((1,1)\)-forms on \(X\) such that \([\rho]=[\eta]\) and \(\rho|_{X_b}=\eta|_{X_b}\) for all \(b\in B\). Then there exists a smooth function \(f:B\to\R\) such that \[\rho=\eta+i\d\db\pi^*f.\]
\end{lemma}

 If \(\omega_X\) and \(\omega_X'\) are relatively cscK metrics with \([\omega_X]=[\omega_X']\), then there exists a function \(\phi\in\mathcal{K}_E\) such that \(\omega_X'=\omega_X+i\d\db\phi\). Note that \(\omega_X+i\d\db\phi=\omega_X+i\d\db\psi\) if and only if \(\phi\) and \(\psi\) differ by a real constant. Hence \(\K_E/\R\) parametrises all relatively cscK metrics. From Lemma \ref{lem:base_form}, two relatively cscK metrics in the same class have the same restriction to every fibre if and only if they differ by \(i\d\db\pi^*f\) for some \(f\in C^\infty(B)\).
 

Recently, Dervan and Sektnan \cite{DS_osc} have introduced the notion of an \emph{optimal symplectic connection}. This is a partial differential equation on relatively cscK metrics which plays the role for fibrations that the cscK equation does for varieties, and also the Hermite--Einstein equation for vector bundles. In fact, when \(X=\P(V)\) is a projectivised vector bundle the optimal symplectic connection equation reduces to the Hermite--Einstein equation \cite[Proposition 3.17]{DS_osc}. Most importantly, when an optimal symplectic connection exists it gives a canonical choice of relatively cscK metric in a fixed relatively K{\"a}hler class.  To describe this equation, we must introduce some notation.

Firstly, any relatively K{\"a}hler metric \(\omega_X\) defines a smooth splitting of the tangent bundle \(TX\cong\mathcal{V}\oplus\mathcal{H}\), where \(\mathcal{V}\) is the kernel of the projection \(d\pi:TX\to TB\), and \[\mathcal{H}:=\{u\in TX:\omega_X(u,v)=0\text{ for all }v\in\mathcal{V}\}.\] One then obtains a canonical splitting of all tensors on \(X\) into horizontal, vertical, and mixed components. The \emph{symplectic curvature} of \(\omega_X\) is the curvature of the Ehresmann connection \(\mathcal{H}\subset TX\). Explicitly, this is a two-form \(F_\H\) on the base \(B\) with values in the fibrewise hamiltonian vector fields, defined by \[F_\H(v_1,v_2):=[v_1^\#,v_2^\#]_{\mathcal{V}}.\] Here \(v_1,v_2\) are vector fields on the base \(B\), \(v_i^\#\) denotes the horizontal lift of \(v_i\) to \(X\), \([\,,]\) is the usual Lie bracket of vector fields, and the subscript \(\mathcal{V}\) denotes the vertical component of a vector field. 

Given a hamiltonian vector field \(\xi_b\) on a fibre \(X_b\), we denote by \(\nu_b^*\xi_b\) the corresponding hamiltonian function on \((X_b,\omega_b)\) with integral zero. If \(\xi\) is a vertical vector field on \(X\) that is fibrewise hamiltonian, we then write \(\nu^*\xi\) for the smooth function on \(X\) equal to \(\nu_b^*\xi_b\) on each fibre \(X_b\). Since \(F_{\mathcal{H}}\) is a 2-form on \(B\) with values in the fibrewise hamiltonian vector fields, \(\nu^*F_\H\) is then a 2-form on \(B\) with values in the fibrewise hamiltonian functions, which can be equivalently considered has a horizontal 2-form on \(X\).

Given a 2-form \(\alpha\) on \(X\), we define \[\Lambda_{\omega_B}\alpha:=\frac{n\,\alpha_{\H}\wedge\pi^*\omega_B^{n-1}}{\pi^*\omega_B^n},\] where \(\alpha_\H\) is the horizontal component of \(\alpha\) and the division is taken in \(\det\H=\Lambda^n\H\).

The metric \(\omega_X\) also determines a vertical Laplacian \(\Delta_{\mathcal{V}}\), defined by \[(\Delta_\V f)|_{X_b}:=\Delta_b(f|_{X_b})\] where \(\Delta_b\) is the K{\"a}hler Laplacian on the fibre \((X_b,\omega_b)\).

Lastly, as \(\omega_X\) is relatively K{\"a}hler it determines a hermitian metric on the vertical tangent bundle \(\mathcal{V}\). This in turn determines a hermitian metric on the determinant bundle \(\Lambda^m\mathcal{V}\cong-K_{X/B}\), whose curvature we denote by \(\rho\).

\begin{definition}[\cite{DS_osc}]\label{def:OSC}
	A relatively cscK metric \(\omega_X\) is an \emph{optimal symplectic connection} if it satisfies the equation \begin{equation}\label{eq:OSC}
	p(\Delta_\mathcal{V}\Lambda_{\omega_B}\nu^*F_\H+\Lambda_{\omega_B}\rho)=0,
	\end{equation} where \(p\) denotes the \(L^2\)-orthogonal projection to \(C^\infty(E)\subset C^\infty(X)\) with respect to the inner product \eqref{eq:inner_product}, \(F_\H\) is the symplectic curvature of the metric \(\omega_X\), and \(\rho\) is the curvature of the hermitian metric on \(\Lambda^m\mathcal{V}\) determined by \(\omega_X\).
\end{definition}

We will often denote by \(\theta\) the smooth function \(\Delta_\mathcal{V}\Lambda_{\omega_B}\nu^*F_\H+\Lambda_{\omega_B}\rho\) appearing in the optimal symplectic connection equation. Hence the optimal symplectic connection equation may be written as \(p(\theta)=0\). We remark that all the symbols we have defined that appear in this equation depend on \(\omega_X\), including the projection operator \(p\) and the contraction operator \(\Lambda_{\omega_B}\).

\begin{definition}\label{def:pi_aut}
	An \emph{automorphism of \(\pi:X\to B\)} is a biholomorphism \(g:X\to X\) satisfying \(\pi\circ g=\pi\). We let \(\Aut(\pi)\) be the complex Lie group of automorphisms of \(\pi\), and let \(\Aut_0(\pi)\) be its identity component.
\end{definition}

\begin{remark}
	One can easily see that if \(\omega_X\) is an optimal symplectic connection then so is \(\omega_X+\pi^*i\d\db f\) for any \(f\in C^\infty(B)\). To see this, note that every term of the optimal symplectic connection equation is unchanged by adding \(\pi^*i\d\db f\), since the horizontal subbundle satisfies \begin{align*}
	\mathcal{H}:&=\{u\in TX:\omega_X(u,v)=0\text{ for all }v\in\mathcal{V}\} \\
	&= \{u\in TX:(\omega_X+\pi^*i\d\db f)(u,v)=0\text{ for all }v\in\mathcal{V}\},
	\end{align*} and the vertical component of \(\omega_X\) agrees with that of \(\omega_X+\pi^*i\d\db f\). All terms in the optimal symplectic connection equation are derived from these two pieces of information. 
	
	Note also that if \(g\in\Aut(\pi)\) then \(g^*\omega_X\) is also an optimal symplectic connection. If \(g\in\Aut_0(\pi)\) is in the identity component, then \(g^*\omega_X\) will be in the same relatively K{\"a}hler class as \(\omega_X\).
\end{remark}

Dervan and Sektnan have proven various results which show these metrics are indeed canonical, including a description of the automorphism group \(\Aut_0(\pi)\) of \(\pi:X\to B\) analogous to the Matsushima--Lichnerowicz theorem \cite[Theorem 4.3, Corollary 4.4]{DS_uniqueness}. Here, we are especially interested in the following uniqueness result: \begin{theorem}[{\cite[Theorem 1.1]{DS_uniqueness}}]\label{thm:OSC_uniqueness}
	If \(\omega_X\) and \(\omega_X'\) are two optimal symplectic connections in the same relatively K{\"a}hler class, then there exist \(g\in\Aut_0(\pi)\) and \(f\in C^\infty(B)\) such that \[\omega_X'=g^*\omega_X+\pi^*i\d\db f.\]
\end{theorem}

Dervan and Sektnan achieve this result by deforming to already known results for twisted cscK metrics. Using the method of geodesics, we will be able to provide a much simplified proof of this result, analogous to the proofs of Donaldson for uniqueness of Hermite--Einstein metrics \cite{Don85} and of Berman--Berndtsson for uniqueness of cscK metrics \cite{BB17}.

The analogue of the Mabuchi functional for the setting of fibrations is given by the following:

\begin{proposition}[{\cite[Lemma 4.28]{DS_uniqueness}}]\label{prop:log-norm}
	There exists a well-defined functional \(\N:\K_E\to\R\) whose derivative along any smooth path \(\varphi:[0,1]\to\K_E\) is given by \[\frac{d\mathcal{N}(\varphi_t)}{dt}=-\int_X\dot{\varphi}_tp_t(\theta_t)\,\omega_t^m\wedge\omega_B^n.\] Here \(p_t(\theta_t)\) is the left-hand side of equation \eqref{eq:OSC} for \(\omega_t:=\omega_X+i\d\db\varphi_t\).
\end{proposition}

The critical points of \(\N\) correspond to optimal symplectic connections in the class \([\omega_X]\). The minus sign in the definition ensures that the functional will be convex along smooth geodesics, rather than concave. 

Our techniques will allow us to generalise the following result of Dervan--Sektnan to arbitrary elements \(\varphi\) of \(\K_E\); see Theorem \ref{thm:boundedness} below for the generalisation.

\begin{theorem}[{\cite[Theorem 4.22]{DS_uniqueness}}]\label{thm:DS_boundedness}
	Suppose there exists an optimal symplectic connection \(\omega_X+i\d\db\psi\) in the class \([\omega_X]\). For any \(\varphi\in\K_E\) that is invariant under the isometry group of \(\omega_X+i\d\db\psi\), \[\N(\varphi)\geq\N(\psi).\]
\end{theorem}

\subsection{Stability of fibrations}\label{sec:fibration_stability}

In this section, we consider a flat projective morphism \(\pi:X\to B\) of complex projective varieties. We let \(L\) be an ample line bundle on \(B\) and \(H\) be a relatively ample line bundle with respect to \(\pi\). In the previous section we assumed that each fibre \(X_b\) admitted a cscK metric. In this section, we shall assume that the general fibre \((X_b,H_b)\) is K-polystable, where \(H_b:=H|_{X_b}\). There is a notion of stability for such fibrations introduced by Dervan and Sektnan \cite{DS_moduli}, which we shall now describe.

\begin{definition}
	Let \(\pi:(X,H)\to(B,L)\) be a fibration. A \emph{fibration degeneration} \((\X,\H)\to(B,L)\) of \((X,H)\to(B,L)\) consists of: \begin{enumerate}[label=(\roman*)]
		\item a variety \(\X\) together with a projective morphism \(\X\to B\times\C\) that has connected fibres, is flat over \(\C\), and is \(\C^*\)-equivariant over \(\C\),
		\item a \(\C^*\)-equivariant line bundle \(\H\to\X\) that is relatively ample over \(B\times\C\),
		\item A fixed isomorphism \((\X_1,\H_1)\cong(X,H^r)\) as fibrations over \(B\), for some \(r>0\) called the \emph{exponent} of the fibration degeneration.
	\end{enumerate}
\end{definition}

\begin{example}
	We can produce many fibration degenerations as follows. Since the map \(\pi:X\to B\) is flat and \(H\) is relatively ample, for all \(k\gg0\) the sheaf push-forward \(\pi_*H^k\) is a vector bundle on \(B\) whose fibre over \(b\in B\) is naturally isomorphic to \(H^0(X_b,H_b^k)\). Since \(H\) is relatively ample, there is then a natural embedding of \(X\) into the projectivised bundle \(\P((\pi_*H^k)^*)\to B\).
	
	For fixed \(k\gg0\), write \(V:=(\pi_*H^k)^*\) so that \(X\subset\P(V)\). Let \(W\subset V\) be a saturated coherent subsheaf.\footnote{Recall a subsheaf \(W\subset V\) is \emph{saturated} if the quotient sheaf \(V/W\) is torsion-free.} This gives rise in a well-known way to a degeneration of \(V\) into the splitting \(W\oplus(V/W)\). That is, there is a sheaf \(\mathcal{V}\) on \(B\times\C\), flat and \(\C^*\)-equivariant over \(\C\), whose general fibre is \(V\) and whose fibre over \(0\in\C\) is \(W\oplus(V/W)\); see, e.g. \cite[Lemma 1.10]{HK18}. Note that \(\mathcal{V}\) is not necessarily flat over \(B\), but is flat over \(\C\).
	
	Since \(X\subset\P(V)=\P(\mathcal{V}|_1)\), we can take the closure \(\overline{\C^*X}\subset\P(\mathcal{V})\) to obtain a scheme \(\X\subset\P(\mathcal{V})\) which is the total space of a fibration degeneration for \(X\). The line bundle \(\H\) is simply given by restricting the \(\mathcal{O}(1)\)-line bundle on \(\P(\mathcal{V})\) to \(\X\).
\end{example}

We will often denote fibration degenerations simply by \((\X,\H)\). As for test configurations, there are certain special classes of fibration degenerations. If there is a \(\C^*\)-equivariant isomorphism \(\X\cong X\times\C\) over \(B\times\C\) with \(X\times\C\) having the trivial \(\C^*\)-action, we call the fibration degeneration \emph{trivial}. If \(\X\cong X\times \C\) with the \(\C^*\)-action induced from a one-parameter subgroup of \(\Aut_0(\pi)\), we call \(\X\) a \emph{product fibration degeneration}.

Note that for every \(j\gg0\), we have a test configuration \((\X,jL+\H)\) for \((X,jL+H^k)\); here we have used additive notation for tensor products of line bundles, and omitted pullbacks. The Donaldson--Futaki invariant then admits an expansion \begin{equation}\label{eq:DF_expansion}
\DF(\X,jL+\H)=j^nW_0(\X,\H)+j^{n-1}W_1(\X,\H)+\O(j^{n-2}),
\end{equation} which is most easily observed from the intersection theory formula for the Donaldson--Futaki invariant.

\begin{definition}[{Dervan--Sektnan \cite{DS_moduli}}]
	The fibration \(\pi:(X,H)\to(B,L)\) is: \begin{enumerate}[label=(\roman*)]
		\item \emph{semistable} if \(W_0(\X,\H)\geq0\) for all fibration degenerations and \(W_1(\X,\H)\geq0\) whenever \(W_0(\X,\H)=0\),
		\item \emph{polystable} if it is semistable and whenever \(W_0(\X,\H)=W_1(\X,\H)=0\), there exists an open subset \(U\subset B\) whose complement has codimension at least 2, such that \((\X,\H)|_U\) normalises to a product fibration degeneration over \(U\),
		\item \emph{stable} if it is semistable and whenever \(W_0(\X,\H)=W_1(\X,\H)=0\), there exists an open subset \(U\subset B\) whose complement has codimension at least 2, such that degeneration \((\X,\H)|_U\) is trivial,
		\item \emph{unstable} if it is not semistable.
	\end{enumerate}
\end{definition}

\begin{remark}
	The definitions of stability and polystability here are slightly different to those in \cite{DS_moduli}. The reason for this is to exclude certain counterexamples due to M. Hattori from preventing any fibration from being polystable. Hattori's example is to take a non-trivial product fibration degeneration of \(X\) and to blow up a single fixed point in the central fibre \(\X_0\). If \(B\) has dimension at least 2, this will leave the invariants \(W_0\) and \(W_1\) of the degeneration, and the minimum norm considered by Dervan--Sektnan \cite[Definition 2.27]{DS_moduli} unchanged. It follows that a fibration with nontrivial automorphisms can never be polystable with respect to the definition of Dervan--Sektnan. Including the codimension 2 hypotheses in the definitions circumvents this kind of counterexample, since two degenerations that are isomorphic away from codimension 2 will have the same invariants \(W_0\) and \(W_1\). This invariance will be made clear from the intersection theoretic formulae in Proposition \ref{prop:W_1_int_theory}, as they all involve intersections with either \(L^n\) or \(L^{n-1}\). The definition is inspired by the vector and principal bundle setting, where it suffices to define stability by considering reductions of the structure group away from codimension 2; see \cite[Definition 1.1]{AB01} for example.
\end{remark}

Given our assumption that the general fibre \((X_b,H_b)\) is K-polystable, the condition that \(W_0(\X,\H)\geq0\) in the definition of stability is in fact superfluous:

\begin{proposition}[{\cite[Lemma 2.29]{DS_moduli}}]\label{prop:flat_locus}
	Let \((\X,\H)\) be a fibration degeneration of \((X,H)\). Then for general \(b\in B\), \[W_0(\X,\H)=\binom{m+n}{n}L^n\cdot\DF(\X_b,\H_b).\]
\end{proposition} As we have assumed the general fibre of \(X\to B\) is K-polystable, it is automatic that \(\DF(\X_b,\H_b)\geq0\) for general \(b\). Hence \(W_0(\X,\H)\geq0\) for any fibration degeneration. Note that if \(W_0(\X,\H)=0\) then the general fibre \(\X_b\) of \(\X\) normalises to a product test configuration for \(X_b\) by definition of K-polystability.

For \(k\gg0\), we define \[\mu_k(X,H):=\frac{-(m+n)K_X\cdot(kL+H)^{m+n-1}}{(kL+H)^{m+n}}.\] There is then an expansion \begin{equation}\label{eq:slope_expansion}
\mu_k(X,H)=A_0(X,H)+A_1(X,H)k^{-1}+\O(k^{-2}),
\end{equation} where \begin{align*}
A_0(X,H) &= \frac{-mK_X\cdot L^n\cdot H^{m-1}}{L^n\cdot H^m},\\
A_1(X,H) &= \frac{(-nK_X\cdot L^{n-1}\cdot H^m)(L^n\cdot H^m)-\frac{m}{m+1}(-K_X\cdot L^n\cdot H^{m-1})(nL^{n-1}\cdot H^{m+1})}{(L^n\cdot H^m)^2}.
\end{align*} 

As for test configurations, we may \(\C^*\)-equivariantly trivialise \((\X,\H)|_{\C^*}\cong(X\times\C^*,p^*H)\), and compactify trivially over \(\P^1\) to obtain \((\overline{\X},\overline{\H})\).

\begin{proposition}[{\cite[pp.~16--17]{DS_moduli}}]\label{prop:W_1_int_theory}
	Let \((\X,\H)\) be a normal fibration degeneration for \((X,H)\), compactified trivially over \(\P^1\). Then \[W_0(\X,\H)=\binom{m+n}{n}\left( \frac{1}{m+1}A_0(X,H)(L^n\cdot\overline{\H}^{\,m+1})+(L^n\cdot\overline{\H}^{\,m}\cdot K_{\overline{\X}/B\times\P^1}) \right)\] and \[W_1(\X,\H)=\binom{m+n}{n-1}(C_1(\X,\H)+C_2(\X,\H)+C_3(\X,\H)),\] where \begin{align*}
	C_1(\X,\H) &= \frac{1}{m+2}A_0(X,H) (L^{n-1}\cdot\overline{\H}^{\,m+2}), \\
	C_2(\X,\H) &= \frac{1}{n}A_1(X,H)(L^n\cdot\overline{\H}^{\,m+1}), \\
	C_3(\X,\H) &= (L^{n-1}\cdot\overline{\H}^{\,m+1}\cdot K_{\overline{\X}/\P^1}).
	\end{align*}
\end{proposition}

Conjecturally, when \(X\to B\) is a holomorphic surjective submersion the existence of an optimal symplectic connection in \(c_1(H)\) is equivalent to polystability of \(\pi:(X,H)\to(B,L)\). A partial result in this direction has been obtained by Dervan and Sektnan: if \(c_1(H)\) admits an optimal symplectic connection then the fibration \((X,H)\to(B,L)\) is semistable \cite[Theorem 1.1]{DS_moduli}.

\section{The space of relatively cscK metrics}\label{sec:geodesic_equation}

In this section we derive the geodesic equation for relatively K{\"a}hler metrics, and prove the existence of a unique smooth geodesic joining any two relatively cscK metrics.  We finish by discussing the structure of the space of relatively cscK metrics as an infinite dimensional symmetric space of non-positive curvature.

\subsection{The geodesic equation}

Fix a relatively K{\"a}hler metric \(\omega_X\) on \(\pi:X\to B\), not assumed to be relatively cscK. Denote by \(\K\) the set of smooth functions \(\varphi\in C^\infty(X)\) such that \(\omega_X+i\d\db\varphi\) is also relatively K{\"a}hler. Recall that if \(\omega_X\) is relatively cscK, then \(\K_E\) denotes the space of smooth functions \(\varphi:X\to\R\) such that \(\omega_X+i\d\db\varphi\) is also relatively cscK, and we have \(\K_E\subset\K\). 

Let \(\varphi:[0,1]\to\mathcal{K}\) be a smooth path. We define the \emph{energy} of \(\varphi\) to be \[\mathcal{Q}(\varphi):=\int_0^1\int_X\dot{\varphi}_t^2\,\omega_t^m\wedge\omega_B^n\,dt,\] where \(\omega_t:=\omega_X+i\d\db\varphi_t\). The path \(\varphi\) is called a \emph{geodesic} if it is a critical point for the energy functional \(\mathcal{Q}\).

\begin{proposition}
	A smooth path \(\varphi:[0,1]\to\mathcal{K}\) is a geodesic if and only if it satisfies the equation \[\ddot{\varphi}_t=|\d_{\mathcal{V}}\dot{\varphi}_t|^2_{\mathcal{V},t}\] for all \(t\in[0,1]\).
\end{proposition}

Here we define \(\d_\V f\) by restricting \(f\) to a fibre \(X_b\) and then taking \(\d (f|_{X_b})\in\Lambda^{1,0}(T^*X_b)\). The hermitian metric \(|\cdot|_{\V,t}\) on \(T^*X_b\) is that determined by \(\omega_{b,t}:=\omega_t|_{X_b}\). Note that the above geodesic equation is simply the usual geodesic equation for K{\"a}hler metrics but fibrewise. That is, \(\varphi\) is a geodesic if and only if \(\varphi|_{X_b}\) is a geodesic in the space of K{\"a}hler potentials for \((X_b,\omega_b)\) for all \(b\in B\).

\begin{proof}
We follow closely the argument in \cite[Proposition 4.25]{Sze14}. Suppose we vary \(\varphi\) by a path \(\psi\) with \(\psi_0=\psi_1=0\). Then writing \(\omega_{\varphi_t+\epsilon\psi_t}:=\omega+i\d\db(\varphi_t+\epsilon\psi_t)\), we have \[\left.\frac{d}{d\epsilon}\right|_{\epsilon=0}\omega_{\varphi_t+\epsilon\psi_t}^m\wedge\omega_B^n=m\,i\d\db\psi_t\wedge\omega_t^{m-1}\wedge \omega_B^n=(\Delta_{\mathcal{V},t}\psi_t)\, \omega_t^m\wedge\omega_B^n,\] where \(\Delta_{\mathcal{V},t}\) is the vertical K{\"a}hler Laplacian with respect to \(\omega_t\). The variation of the energy is then \begin{align*}
\left.\frac{d}{d\epsilon}\right|_{\epsilon=0}\mathcal{Q}(\varphi+\epsilon\psi)=\int_0^1\int_X(2\dot{\varphi}_t\dot{\psi}_t+\dot{\varphi}_t^2(\Delta_{\mathcal{V},t}\psi_t))\,\omega_t^m\wedge\omega_B^n\,dt.
\end{align*} Splitting the integrand into two terms, for the second term we have \begin{align*}
\int_X\dot{\varphi}_t^2(\Delta_{\mathcal{V},t})\,\omega_t^m\wedge\omega_B^m &= \int_B\left(\int_{X/B}\dot{\varphi}_t^2(\Delta_{\mathcal{V},t}\psi_t)\,\omega_t^m\right)\omega_B^n \\
&= \int_B\left(\int_{X/B}\Delta_{\mathcal{V},t}(\dot{\varphi}_t^2)\psi_t\,\omega_t^m\right)\omega_B^n \\
&= \int_X \psi_t\Delta_{\mathcal{V},t}(\dot{\varphi}_t^2)\,\omega_t^m\wedge\omega_B^n.
\end{align*} For the first term, consider \begin{align*}
\frac{d}{dt}(\dot{\varphi}_t\psi_t\,\omega_t^m\wedge\omega_B^n)
&= (\ddot{\varphi}_t\psi_t+\dot{\varphi}_t\dot{\psi}_t+\dot{\varphi}_t{\psi}_t(\Delta_{\mathcal{V},t}\dot{\varphi}_t))\,\omega_t^m\wedge\omega_B^n.
\end{align*} Integrating this first over \(X\) and then over \(t\in[0,1]\), the left hand side vanishes since \(\psi_0=\psi_1=0\). Hence the first term in the variation of the energy becomes \[\int_0^1\int_X-2\psi_t(\ddot{\varphi}_t+\dot{\varphi}_t(\Delta_{\mathcal{V},t}\dot{\varphi}_t))\,\omega_t^m\wedge\omega_B^n\,dt.\] Putting this all together, the variation is equal to \[\int_0^1\int_X\psi_t[-2\ddot{\varphi}_t-2\dot{\varphi}_t(\Delta_{\mathcal{V},t}\dot{\varphi}_t)+\Delta_{\mathcal{V},t}(\dot{\varphi}_t^2)]\,\omega_t^m\wedge\omega_B^n\,dt.\]
Using the fact that \[\Delta_{\mathcal{V},t}(\dot{\varphi}_t^2)=2\dot{\varphi_t}(\Delta_{\mathcal{V},t}\dot{\varphi}_t)+2\langle\d_{\mathcal{V}}\dot{\varphi}_t\, ,\d_{\mathcal{V}}\dot{\varphi}_t\rangle_{\mathcal{V},t},\] where \(\d_{\mathcal{V}}\) is the vertical \(\d\)-operator and \(\langle\, ,\rangle_{\mathcal{V},t}\) is the vertical inner product given by restriction of \(\omega_t\) to the fibres, the variation is equal to \[\int_0^1\int_X\psi_t(-2\ddot{\varphi}_t+2|\d_{\mathcal{V}}\dot{\varphi}_t|^2_{\mathcal{V},t})\,\omega_t^m\wedge\omega_B^n\,dt.\] Hence the geodesic equation is \[\ddot{\varphi}_t=|\d_{\mathcal{V}}\dot{\varphi}_t|^2_{\mathcal{V},t}.\qedhere\] 
\end{proof}

\begin{remark}
	If \(\omega_X\) is relatively cscK, the above proof shows that a smooth path \(\varphi_t:[0,1]\to\K_E\) is a geodesic if and only if it is a geodesic in the ambient space \(\K\); the same calculation holds if we restrict to variations of paths within \(\K_E\).
\end{remark}

\subsection{Existence of geodesics}

We now work towards proving the following existence result: \begin{theorem}\label{thm:geodesic_existence}
	Let \(\omega_X\) be a relatively cscK metric. For any \(\varphi_0,\varphi_1\in\K_E\), there exists a unique geodesic \(\varphi:[0,1]\to \K_E\) joining \(\varphi_0\) to \(\varphi_1\).
\end{theorem}

We will assume \(\varphi_0=0\) for simplicity, and write \(\omega_X':=\omega_X+i\d\db\varphi_1\). For each \(b\in B\), denote by \(E_b\) the space of real holomorphy potentials on \((X_b,\omega_b)\) having zero integral. Recalling Proposition \ref{prop:holomorphy_geodesics}, on any fibre \(X_b\) there exists a unique \(u_b\in E_b\) which relates the metrics \(\omega_b\) and \(\omega_b'\) in the following way. Letting \(v_b:=\nabla^{1,0}_bu_b\) be the holomorphic vector field generated by \(u_b\), and letting \(\exp(tv_b)\) be the time \(t\) flow of the real part of \(v_b\), we have that \begin{equation}\label{eq:dummy2}
\psi_{b,t}:=\int_0^t\exp(sv_b)^*u_b\,ds
\end{equation} is a smooth geodesic from \(0\) to \(\psi_{b,1}=\varphi_1|_{X_b}\) in the space \(\K_b\) of K{\"a}hler potentials for \((X_b,\omega_b)\), and \(\omega_b'=\omega_b+i\d\db\psi_{b,1}\). 

Let us now \emph{assume} that the \(u_b\) vary smoothly with \(b\), so as to define a global smooth function \(u:X\to\R\). Then the \(\psi_{b,t}\) also depend smoothly on \(b\) and \(t\), and so define a global smooth geodesic \(\psi:[0,1]\to\K_E\).  If we let \[\omega_{X,t}:=\omega_X+i\d\db\psi_t\] then we have a smooth path of relatively cscK metrics such that \(\omega_{X,0}=\omega_X\) and \(\omega_{X,1}|_{X_b}=\omega_X'|_{X_b}\) for any \(b\in B\). By Lemma \ref{lem:base_form} there exists a smooth function \(f:B\to\R\) such that \(\omega_{X,1}-\omega_X'=i\d\db\pi^*f\). Since \[\omega_X+i\d\db\psi_1=\omega_X+i\d\db(\varphi_1+f),\] by adding a constant to \(f\) we may assume that \(\varphi_1=\psi_1-f\). Noting that \(\d_{\V}f=0\), the path \(\varphi_t:=\psi_t-tf\) is a smooth geodesic in \(\K_E\) joining \(\varphi_0=0\) to \(\varphi_1\).

Thus, for the existence of smooth geodesics it only remains to show that the holomorphy potentials \(u_b\in E_b\) depend smoothly upon \(b\). To prove this, we will introduce two fibre bundles: a smooth fibre bundle \(H\to B\) whose fibre \(H_b\) is the homogeneous space of cscK metrics on \((X_b,[\omega_b])\), and a fibre bundle \(\K_0\) whose fibre \(\K_{0,b}\subset\K_b\) consists of suitably normalised K{\"a}hler potentials on \((X_b,\omega_b)\). The idea will be to show that \(H\) embeds as a smooth subbundle of \(\K_0\), so that a smooth choice of fibrewise cscK metric will then determine a smooth section of \(H\).

Given the fixed relatively cscK metric \(\omega_X\), for each \(b\in B\) write \(G_b:=\Aut_{\red}(X_b)\) and let \(K_b:=\mathrm{Isom}_0(\omega_b)\cap G_b\) be the maximal compact subgroup of isometries of \(\omega_b\) in \(G_b\). By the proof of Proposition \ref{prop:holomorphy_geodesics}, we have a diffeomorphism \(E_b\cong G_b/K_b\), and the homogeneous space \(G_b/K_b\) parametrises the cscK metrics on \(X_b\) in the class \([\omega_b]\). Letting \(H\) be the disjoint union of the \(H_b\) over \(b\in B\), we have a set-theoretic bijection \(H\cong E\). Proposition \ref{prop:bundle_E} tells us that \(E\) is a smooth vector bundle, so we may endow \(H\) with a smooth structure by declaring the set-theoretic bijection between \(H\) and \(E\) to be a diffeomorphism.

Let us now consider the non-relative case, so we have a compact K{\"a}hler manifold \((Y,\omega)\) of dimension \(d\), with space of K{\"a}hler potentials \(\K\). For any \(\phi\in\K\), there is a splitting of the tangent space \[T_\phi\K=\left\lbrace\psi\in C^\infty(Y):\int_Y\psi\,\omega_\phi^d=0\right\rbrace\oplus\R,\] where \(\R\) denotes the constant functions. This splitting integrates to a decomposition \[\K\cong\K_0\times\R,\] see \cite[Section 2.4]{Che00}. In fact, \(\K_0=I^{-1}(0)\), where \(I:\K\to\R\) is the \emph{Monge--Amp{\`e}re energy} (also called the \emph{Aubin--Mabuchi functional}): \[I(\phi):=\frac{1}{d+1}\sum_{j=0}^d\int_Y\phi\,\omega_\phi^j\wedge\omega^{d-j}.\] Since \(I(\phi+c)=I(\phi)+c\), for any \(\omega'\in[\omega]\) there exists a unique \(\phi\in\K_0\) such that \(\omega'=\omega_\phi\).

Now, consider the above splitting for the fibre \((X_b,\omega_b)\): let \(\K_b\) to be the space of K{\"a}hler potentials for \((X_b,\omega_b)\) with decomposition \(\K_b\cong\K_{0,b}\times\R\). This decomposition depends smoothly upon \(b\) since the functionals \(I_b:\K_b\to\R\) do, and so there is a smooth Fr{\'e}chet bundle \(\K_0\to B\) whose fibre over \(b\) is \(\K_{0,b}\). Notice by \eqref{eq:integral_zero} of Proposition \ref{prop:holomorphy_geodesics} that the path of K{\"a}hler potentials \(\psi_{b,t}\) defined in \eqref{eq:dummy2} lies within \(\K_{0,b}\). The following Lemma will tell us that \(\psi_{b,1}\) is smooth in \(b\).

 \begin{lemma}\label{lem:bundle_isom}
	Let \(\omega_X'\) be a relatively K{\"a}hler metric, and for each \(b\in B\) denote by \(\phi_b\) the unique K{\"a}hler potential in \(\K_{0,b}\) such that \(\omega_b'=\omega_b+i\d\db\phi_b\). Then the \(\phi_b\) depend smoothly upon \(b\), so piece together to a global smooth section \(\phi:B\to\K_0\).
\end{lemma}

\begin{proof}
	Write \(\K^k\) for the Banach fibre bundle on \(B\) whose fibre \(\K_b^k\) consists of \(C^k\)-regular K{\"a}hler potentials for \((X_b,\omega_b)\). The fibrewise functionals \(I_b:\K_b^k\to \R\) piece together to define a global smooth map \(I:\K^k\to B\times\R\) whose derivative is everywhere surjective. By the implicit function theorem for Banach manifolds, the inverse image \(I^{-1}(0)\) is a smooth Banach subbundle \(\K^k_0\subset\K^k\), consisting of the normalised \(C^k\)-regular K{\"a}hler potentials on each fibre.
	
	Similarly, denote by \(\L^{k-2}\) the Banach fibre bundle on \(B\) whose fibre \(\L^{k-2}_b\) is the space of \(C^{k-2}\)-regular K{\"a}hler metrics in the class \([\omega_b]\). There is a smooth morphism of Banach fibre bundles \(\K^k_0\to\L^{k-2}\) given over \(b\in B\) by \[\psi_b\mapsto \omega_b+i\d\db\psi_b.\] The derivative of this morphism on any fibre is an isomorphism, hence by the inverse function theorem for Banach manifolds, we have an isomorphism \(\K^k_0\cong\L^{k-2}\).
	
	For any \(k\), the relatively cscK metric \(\omega_X'\) determines a smooth section of \(\L^{k-2}\). It follows that the corresponding section \(\phi\) of normalised K{\"a}hler potentials is a smooth section of \(\K^k_0\) for any \(k\). Hence the \(\phi_b\) depend smoothly upon \(b\).
\end{proof}

 Letting \(\L\) denote the fibre bundle whose fibre \(\L_b\) is the space of K{\"a}hler metrics in \([\omega_b]\) as in the above proof, we have the following commutative diagram of morphisms of fibre bundles over \(B\): \begin{equation}\label{eq:diagram}
 \begin{tikzcd}
   	H \arrow[r,"\sim"]\arrow[d]&\arrow[l]\arrow[d] E \\
   	\L\arrow[r,"\sim"] & \K_0 \arrow[l]
   	\end{tikzcd}.
 \end{equation} The right-most vertical arrow \(E\to\K_0\) is simply defined by composing the other three arrows. The smooth map of fibre bundles \(H\to\L\) is given by \[(b,K_bg)\mapsto (b,g^*\omega_b).\] There are right \(G_b\)-actions on \(H_b\) and \(\L_b\), and the morphism \(H_b\to\L_b\) is \(G_b\)-equivariant by construction. We can therefore consider the associated morphism \(E_b\to\K_{0,b}\) as \(G_b\)-equivariant as well.
 
 \begin{theorem}\label{prop:universal_property}
 	The map of fibre bundles \(E\to\mathcal{K}_0\) in the diagram \eqref{eq:diagram} satisfies the following universal property: given any relatively cscK metric \(\omega_X'\) on \(X\) with associated smooth section \(B\to \mathcal{K}_0\), there exists a unique smooth section \(u:B\to E\) such that the following diagram commutes \begin{center}
 		\begin{tikzcd}[column sep = small]
 		E \arrow[rr]&& \mathcal{K}_0 \\
 		& B\arrow[ul,dashed]\arrow[ur] &
 		\end{tikzcd}
 	\end{center} The section \(u:B\to E\) is such that \(u_b\in E_b\) is the unique real holomorphy potential relating \(\omega_b\) and \(\omega_b'\) as in Proposition \ref{prop:holomorphy_geodesics}.
 \end{theorem}

 \begin{proof}
 	Existence and uniqueness of the map as a map of sets is clear. The only question is smoothness. For this, we will show that \(E\to\mathcal{K}_0\) is an injective immersion, and that the preimage of any closed bounded set in \(\K_0\) (with respect to any \(C^k\)-norm) is compact in \(E\). Then smoothness will follow immediately from the commuting diagram.

 	Note that injectivity of \(E\to\mathcal{K}_0\) follows from injectivity of \(H\to\mathcal{L}\). To see that \(E\to\mathcal{K}_0\) is an immersion, it suffices to show that for each \(b\in B\) the map \(E_b\to\mathcal{K}_{0,b}\) is an immersion. We thus drop the subscript \(b\) and work on a single fixed K{\"a}hler manifold \((X,\omega)\) of constant scalar curvature. Given \(u\in E\), let \(v_u:=\nabla^{1,0}u\) be the associated holomorphic vector field. The map \(E\to\mathcal{K}_0\) is then given by \[u\mapsto\int_0^1\exp(tv_u)^*u\,dt.\] We will show that the derivative of this map is injective at the origin. By \(G\)-equivariance of the map \(E\to \K_0\) and transitivity of the \(G\)-action on \(E\cong G/K\), this will imply the derivative is everywhere injective. We compute \[\left.\frac{d}{ds}\right|_{s=0}\int_0^1\exp(stv_u)^*(su)\,dt=\int_0^1u\,dt=u.\] Hence the derivative is injective at the origin.
 	
 	To show the inverse image of any closed bounded set is compact, we take a maximum \(x\in X_b\) of \(u\in E_b\backslash\{0\}\). Then \(\nabla^{1,0}_bu(x)=0\), and the flow \(\exp(v_u)\) fixes \(x\). Since \(u(x)>0\), \[\int_0^t\exp(sv_u)^*u(x)\,ds=\int_0^tu(x)\,ds=tu(x)\to\infty\] as \(t\to\infty\). A compactness argument on the unit sphere bundle of \(E\) shows this divergence is uniform over compact subsets, and so the preimage of a closed bounded set must be compact.
 \end{proof}

 \begin{proof}[Proof of Theorem \ref{thm:geodesic_existence}]
 	We have shown that a geodesic exists, since the previous theorem shows the \(u_b\) depend smoothly upon \(b\). For uniqueness, note that we already have fibrewise uniqueness from the usual theory of K{\"a}hler geodesics. Thus, two geodesics could only potentially differ by a smooth path \(\psi_t\) in \(C^\infty(B)\). Since \(\d_{\V}\dot{\psi}_t=0\), by the geodesic equation this path must then satisfy \(\ddot{\psi}_t=0\). Hence \(\psi_t=tf+g\) for some smooth functions \(f,g:B\to\R\). The boundary conditions then imply \(f=g=0\).
 \end{proof}
 
 \subsection{Negative curvature of the space \(\mathcal{K}_E\)}

We will end this section with a discussion of the structure of the space \(\K_E\). First, the following decomposition is clear from the above: \begin{proposition}
	There is a Riemannian decomposition \[\K_E\cong C^\infty(H)\times C^\infty(B),\] where \(H\to B\) is the fibre bundle whose fibre is \(G_b/K_b\) and \(C^\infty(H)\) denotes the smooth sections of \(H\). This decomposition integrates the infinitesimal decomposition of Proposition \ref{prop:K_E_tangent_decomp}.
\end{proposition}

This is the relative version of the decomposition \(\K\cong\K_0\times\R\) of the space of K{\"a}hler potentials described above, and thus the smooth functions on \(B\) play the role of the constant functions in this setting.

One can go further with this; the space of K{\"a}hler potentials \(\K\) on a compact K{\"a}hler manifold \((Y,\omega)\) has a natural connection defined as follows. Let \(\phi:[0,1]\to\K\) be a smooth path, and let \(\psi:[0,1]\to C^\infty(Y)\) be a tangent vector along \(\phi\). The covariant derivative of \(\psi\) along \(\phi\) is then given by \[\nabla_{\dot{\phi}_t}\psi=\dot{\psi}_t-\langle\d\psi_t,\d\dot{\phi}_t\rangle_t.\] Taking \(\psi_t=\dot{\phi}_t\), one recovers the geodesic equation. The curvature of this connection at \(\phi\in\K\) is \[R_\phi(\psi,\eta)\xi=-\frac{1}{4}\{\{\psi,\eta\}_\phi,\xi\}_\phi,\] where \(\psi,\eta,\xi\in T_\phi\K\cong C^\infty(X)\) and \(\{\,,\}_\phi\) is the Poisson bracket of the symplectic form \(\omega_\phi:=\omega+i\d\db\phi\); \cite{Mab87,Sem92,Don99}. The sectional curvature is therefore \[K_\phi(\psi,\eta)=-\frac{1}{4}\|\{\psi,\eta\}_\phi\|^2_\phi,\] so \(\K\) is non-positively curved. Note that each slice \(\{\phi\}\times\R\subset\K_0\times\R\cong\K\) is a flat geodesic submanifold.

The exact same calculations in the relative setting (compare to \cite[Section 4]{Blo13}) show the following: \begin{proposition}\label{prop:curvature}
	There is a connection on \(\K_E\) given by \[\nabla_{\dot{\varphi_t}}\psi=\dot{\psi}_t-\langle\d_{\V}\psi_t,\d_{\V}\dot{\varphi_t}\rangle_{\V,t},\] where \(\varphi:[0,1]\to \K_E\) and \(\psi_t\in T_{\varphi_t}\K_E\) is a tangent vector along \(\varphi_t\). The curvature of this connection at \(\varphi\in\K_E\) is \[R_\varphi(\psi,\eta)\xi=-\frac{1}{4}\{\{\psi,\eta\}_{\V,\varphi},\xi\}_{\V,\varphi},\] where \(\psi,\eta,\xi\in T_\varphi\K_E\subset C^\infty(X)\) and \(\{\,,\}_{\V,\varphi}\) is the fibrewise Poisson bracket corresponding to the relatively symplectic form \(\omega_{X,\varphi}:=\omega_X+i\d\db\varphi\). The sectional curvature is \[K_\varphi(\psi,\eta)=-\frac{1}{4}\|\{\psi,\eta\}_{\V,\varphi}\|^2_\varphi,\] so \(\K_E\) is non-positively curved. The subspace \(C^\infty(B)\subset\K_E\) is a flat totally geodesic submanifold.
\end{proposition}

Note that the fibres \(H_b\) have their own curvatures as Riemannian symmetric spaces: if \(u,v,w\in T_{K_b}H_b=i\cdot\mathfrak{k}_b\) (see Lemma \ref{lem:homog_space}), then \[R_{K_b}(u,v)w=-\frac{1}{4}[[u,v],w],\] where \([\,,]\) is the Lie bracket of the Lie algebra of \(G_b\). Since the isomorphism \(E_b\cong T_{K_b}H_b\) is essentially the hamiltonian vector field construction, this corresponds precisely to the Poisson bracket of functions in \(E_b\). So the natural curvatures of the spaces \(H_b\) gives rise to the curvature on \(C^\infty(H)\) of Proposition \ref{prop:curvature}.

\section{Convexity along geodesics and applications to uniqueness and boundedness}\label{functional}

\subsection{Convexity of \(\mathcal{N}\) along geodesics}

Recall the log-norm functional for fibrations is defined through its first derivative along smooth paths \(\varphi_t\) in \(\K_E\): \[\frac{d}{dt}\mathcal{N}(\varphi_t)=-\int_X\dot{\varphi}_t\,p_t(\theta_t)\,\omega_t^m\wedge\omega_B^n,\] where \(p_t(\theta_t)\) is the optimal symplectic connection operator applied to the relatively cscK metric \(\omega_t:=\omega_X+i\d\db\varphi_t\), as defined in equation \eqref{eq:OSC}. \begin{theorem}[Convexity]\label{thm:fib_convexity}
	The functional \(\mathcal{N}:\mathcal{K}_E\to\R\) is convex along smooth geodesics in \(\K_E\). If \(\varphi_t\) is a smooth geodesic in \(\K_E\), the function \(t\mapsto\N(\varphi_t)\) is affine linear if and only if there exist \(u\in C^\infty(E)\) and \(f\in C^\infty(B)\) such that \(v:=\nabla_\V^{1,0}u\) is a holomorphic vector field, and \[\varphi_t=\psi_t+tf,\] where \(\psi_t\) is defined as in \eqref{eq:dummy2}. Otherwise, \(\mathcal{N}(\varphi_t)\) is strictly convex.
\end{theorem} We note that for \(u\in C^\infty(E)\), the vertical vector field \(\nabla^{1,0}_{\V}u\) is holomorphic on each fibre, but may not be globally holomorphic as is required in the theorem. Before proving this, we must first linearise the optimal symplectic connection operator \(\varphi\mapsto p_\varphi(\theta_\varphi)\), where \(p_\varphi(\theta_\varphi)\) is given by the left-hand side of \eqref{eq:OSC} with respect to the reference metric \(\omega_{X,\varphi}:=\omega_X+i\d\db\varphi\). This is achieved through the following expansion:

\begin{lemma}[{\cite[Proposition 2.4]{DS_moduli}}]\label{lem:expansion}
Let \(\omega_X\) be relatively cscK, and for \(k\gg0\) write \(\omega_k:=\omega_X+k\omega_B\), which is a K{\"a}hler metric on \(X\). Then there exists a smooth function \(\psi_R\in C^\infty_R(X)\) such that there is a \(C^\infty\)-expansion \[S(\omega_k+k^{-1}i\d\db\psi_R)=S(\omega_b)+k^{-1}(S(\omega_B)+p(\theta))+\O(k^{-2}),\] where \(S(\omega_b)\) is the smooth function on \(X\) whose restriction to the fibre \(X_b\) is the scalar curvature of \(\omega_b\), and \(\theta:=\Delta_\V\Lambda_{\omega_B}\nu^*F_\H+\Lambda_{\omega_B}\rho\).
\end{lemma} Notice that since \(\omega_X\) is relatively cscK, the leading term \(S(\omega_b)\) is a topological constant \(\hat{S}(\omega_b)\) independent of \(b\). To see that \(\hat{S}(\omega_b)\) does not depend on \(b\), we can for example note the fibre integral formula \[\hat{S}(\omega_b)=\frac{m\int_{X/B}\rho\wedge\omega^{m-1}}{\int_{X/B}\omega^m}(b).\] Since fibre integrals of closed forms produce closed forms on the base, the resulting functions on \(B\) in the numerator and denominator are closed, hence constant.

For any \(\varphi\in\K_E\), define the operator \(\mathcal{R}_\varphi:C^\infty(X)\to\Omega^{0,1}(\Lambda^{1,0}\V)\) by \[\mathcal{R}_\varphi\psi:=\db\nabla^{1,0}_{\V,\varphi}\psi,\] where \(\nabla_{\V,\varphi}\) is the vertical fibrewise gradient operator of \(\omega_{X,\varphi}\). When \(\psi\in C^\infty(B)\oplus C^\infty(E_\varphi)=T_\varphi\K_E\) is a fibrewise holomorphy potential, \(\nabla^{1,0}_{\V,\varphi}\psi\) is a fibrewise holomorphic vector field on \(X\) that is globally holomorphic if and only if \(\mathcal{R}_\varphi\psi=0\). \begin{proposition}
The linearisation of the operator \(\varphi\mapsto p_\varphi(\theta_\varphi)\) at \(\varphi\in\K_E\) in the direction \(\psi\in C^\infty(B)\oplus C^\infty(E_\varphi)=T_\varphi\K_E\) is given by \[-\mathcal{R}_\varphi^*\mathcal{R}_\varphi\psi+\langle\d_{\mathcal{V}}p_\varphi(\theta_\varphi),\d_{\mathcal{V}}\psi\rangle_{\mathcal{V},\varphi}.\]
\end{proposition}

\begin{proof}
	For ease of notation, we fix \(\omega_{X,\varphi}\) as our new reference metric \(\omega_X\), so that \(\varphi=0\). Consider the expansion \begin{equation}\label{dummy4}
	S(\omega_k+k^{-1}i\d\db\psi_R)=S(\omega_b)+k^{-1}(S(\omega_B)+p(\theta))+\O(k^{-2}),
	\end{equation} from Lemma \ref{lem:expansion}, and write \(\omega_k':=\omega_k+k^{-1}i\d\db\psi_R\). Suppose we vary \(\omega_X\) by \(ti\d\db \psi\), where \(\psi\in C^\infty(B)\oplus C^\infty(E)\), and differentiate at \(t=0\). Then the linearisation of the left hand side of equation \eqref{dummy4} is \begin{equation}\label{dummy?}
	-(\mathcal{D}'_k)^*\mathcal{D}'_k\psi+\langle\d S(\omega_k'),\d\psi\rangle_{\omega_k'}
	\end{equation} where \((\mathcal{D}_k')^*\mathcal{D}_k'\) is the Lichnerowicz operator of \(\omega_k'\); see \cite[Lemma 4.4]{Sze14}.
	
	Consider the term \(\langle\d S(\omega_k'),\d\psi\rangle_{\omega_k'}\). We claim there is an expansion \(\langle\,,\rangle_{\omega_k'}=\langle\,,\rangle_{\omega_k}+\O(k^{-1})\) of metrics on \(T^*X\) as \(k\to\infty\). To see this, note the metric on the cotangent bundle determined by \(\omega_k'\) is given locally by inverting the matrix of \(\omega_k'\) in local coordinates. Thus, at a point \(x\in X\), let us choose holomorphic coordinates \(z_1,\ldots,z_n,w_1,\ldots,w_m\) compatible with the fibration, meaning that the projection \(\pi:X\to B\) is given by \((z_1,\ldots,z_n,w_1,\ldots,w_m)\mapsto(z_1,\ldots,z_n)\) in these coordinates. By making a linear change of coordinates, we may further assume that the \(\d/\d z_j\) span the horizontal distribution \(\mathcal{H}\subset TX\) at the point \(x\). Thus, the metric \(\omega_X+k\omega_B\) on \(TX\) takes the form \[\begin{pmatrix}
	A & 0 \\
	0 & B+kC
	\end{pmatrix}\] at the point \(x\). Here \(A\) is the restriction of \(\omega_X\) to the vertical tangent bundle, \(B\) the restriction of \(\omega_X\) to \(\mathcal{H}\), and \(C\) the matrix of \(\omega_B\) in the coordinates \(z_1,\ldots,z_n\). The inverse \(\omega_k^{-1}\) in local coordinates then has an expansion \begin{align*}
	\begin{pmatrix}
		A & 0 \\
		0 & B+kC
	\end{pmatrix}^{-1}
	&=\begin{pmatrix}
		A^{-1} & 0 \\
		0 & (B+kC)^{-1}
    \end{pmatrix} \\
    &=\begin{pmatrix}
		A^{-1} & 0 \\
		0 & (kC)^{-1}(I+k^{-1}BC^{-1})^{-1}
   	\end{pmatrix} \\
   	&=\begin{pmatrix}
   		A^{-1} & 0 \\
   		0 & (kC)^{-1}(I+\O(k^{-1}))
   	\end{pmatrix} \\
   	&= M+\O(k^{-1}),
	\end{align*} where \[M:=\begin{pmatrix}
	   		A^{-1} & 0 \\
	   		0 & 0
	   	\end{pmatrix}.\] We use this to see \begin{align*}
	(\omega_k')^{-1} &= (\omega_k+k^{-1}i\d\db\psi_R)^{-1} \\
	&= \omega_k^{-1}(I+k^{-1}(i\d\db\psi_R)\omega_k^{-1})^{-1} \\
	&= \omega_k^{-1}(I+k^{-1}(i\d\db\psi_R)(M+\O(k^{-1})))^{-1} \\
	&= \omega_k^{-1}(I+\O(k^{-1}))^{-1} \\
	&= \omega_k^{-1}+\O(k^{-1}),
	\end{align*} as claimed. 
	
	Since \(\d S(\omega_k')\) is also of order \(k^{-1}\) by \eqref{dummy4}, we may instead consider \(\langle\d S(\omega_k'),\d\psi\rangle_{\omega_k}\) in place of \(\langle\d S(\omega_k'),\d\psi\rangle_{\omega_k'}\) in \eqref{dummy?}. Splitting this into its horizontal and vertical components with respect to \(\omega_k\) gives \[\langle\d_{\mathcal{H}}S(\omega_k'),\d_{\mathcal{H}}\psi\rangle_{\omega_k}+\langle\d_{\mathcal{V}}S(\omega_k'),\d_{\mathcal{V}}\psi\rangle_{\omega_k}.\] The horizontal term is uniformly of order \(k^{-2}\); the leading coefficient of \(S(\omega_k)\) is constant, so \(\d_\H S(\omega_k)\) is order \(k^{-1}\), and another \(k^{-1}\) factor comes from the inverse of \(\omega_k\) in the horizontal direction, as we see from the above calculation in local coordinates. From the expansion of \(S(\omega_k')\), the order \(k^{-1}\) part of the vertical term is given by \[\langle\d_{\mathcal{V}}p(\theta),\d_{\mathcal{V}}\psi\rangle_{\mathcal{V}}.\] The order \(k^{-1}\) term of \((\mathcal{D}_k')^*\mathcal{D}_k'\psi\) is given by \(\mathcal{R}^*\mathcal{R}\psi\); see \cite[Section 4.4]{DS_osc}. The linearisation of the optimal symplectic connection operator in the direction \(\psi\) is then \[-\mathcal{R}^*\mathcal{R}\psi+\langle\d_{\mathcal{V}}p(\theta),\d_{\mathcal{V}}\psi\rangle_{\mathcal{V}}.\qedhere\]
\end{proof}

 \begin{proof}[Proof of Theorem \ref{thm:fib_convexity}] We compute \begin{align*}
-\frac{d}{dt}\int_X \dot{\varphi}_t\,p_t(\theta_t)\,\omega_t^m\wedge\omega_B^n 
&= -\int_X [\ddot{\varphi}_t\,p_t(\theta_t)+\dot{\varphi}_t(-\mathcal{R}_t^*\mathcal{R}_t\dot{\varphi}_t+\langle\d_{\mathcal{V}}p_t(\theta_t),\d_{\mathcal{V}}\dot{\varphi}_t\rangle_{\mathcal{V},t}) \\ &\quad\quad\quad\quad\quad\quad\quad\quad\quad\quad\quad\quad\quad +\dot{\varphi}_tp_t(\theta_t)\Delta_{\mathcal{V},t}\dot{\varphi}_t]\,\omega_t^m\wedge\omega_B^n .
\end{align*} Breaking this up, we have \begin{align*}
\int_X\dot{\varphi}_t\langle\d_{\mathcal{V}}p_t(\theta_t),\d_{\mathcal{V}}\dot{\varphi}_t\rangle_{\mathcal{V},t}\,\omega_t^m\wedge\omega_B^n &= \int_B\left(\int_{X/B}\frac{1}{2}\langle\d_{\mathcal{V}}p_t(\theta_t),\d_{\mathcal{V}}(\dot{\varphi}_t^2)\rangle_{\mathcal{V},t}\,\omega_t^m\right)\omega_B^n \\
&= \int_B\left(\int_{X/B}-\frac{1}{2}p_t(\theta_t)\Delta_{\mathcal{V},t}(\dot{\varphi}_t^2)\,\omega_t^m\right)\omega_B^n \\
&= \int_X-\frac{1}{2}p_t(\theta_t)\Delta_{\mathcal{V},t}(\dot{\varphi}_t^2)\,\omega_t^m\wedge\omega_B^n.
\end{align*} Also, \begin{align*}
\int_X \dot{\varphi}_tp_t(\theta_t)\Delta_{\mathcal{V},t}\dot{\varphi}_t\,\omega_t^m\wedge\omega_B^n &= \int_X p_t(\theta_t)\left(\frac{1}{2}\Delta_{\mathcal{V},t}(\dot{\varphi}_t^2)-\langle\d_{\mathcal{V}}\dot{\varphi}_t,\d_{\mathcal{V}}\dot{\varphi}_t\rangle_{\mathcal{V},t}\right)\,\omega_t^m\wedge\omega_B^n.
\end{align*} Hence the second variation of \(\mathcal{N}\) is equal to \begin{align*}
\int_X|\mathcal{R}_t\dot{\varphi}_t|^2\omega_t^m\wedge\omega_B^n-\int_X(\ddot{\varphi}_t-|\d_{\mathcal{V}}\dot{\varphi}_t|^2_{\mathcal{V},t})\,p_t(\theta_t)\,\omega_t^m\wedge\omega_B^n.
\end{align*} If \(\varphi_t\) satisfies the geodesic equation then this is non-negative, and the functional is convex. It is strictly convex unless \(\mathcal{R}_t\dot{\varphi}_t=0\), which holds if and only if the fibrewise-holomorphic vector field \(\nabla^{1,0}_t\dot{\varphi}_t\) is globally holomorphic on \(X\). \end{proof}

\subsection{Applications of convexity to uniqueness and boundedness}

We recall from Definition \ref{def:pi_aut} the group \(\Aut_0(\pi)\), which is the identity component of the group of biholomorphisms \(g:X\to X\) satisfying \(\pi\circ g=\pi\). Using our convexity result, we can give a new proof of Theorem \ref{thm:OSC_uniqueness}, which is due to \cite{DS_uniqueness}.

\begin{proof}[Proof of Theorem \ref{thm:OSC_uniqueness}]
	Suppose \(\omega_X'=\omega_X+i\d\db \varphi_1\) and let \(\varphi:[0,1]\to\K_E\) be the unique geodesic joining \(\varphi_0=0\) to \(\varphi_1\). Let \[\varphi_t=\psi_t+tf\] be the decomposition so that \(\dot{\psi}_t\in C^\infty(E_t)\) for all \(t\) and \(f\in C^\infty(B)\). Writing \(v:=\nabla^{1,0}_{\V}\dot{\psi}_0\), we have \[\psi_t=\int_0^t\exp(sv)^*\dot{\psi}_0\,ds,\] and \(\omega_{X,t}=\exp(sv)^*\omega_X\) where \(\omega_{X,t}:=\omega_X+i\d\db\psi_t\). Since \(\omega_X\) and \(\omega_{X,1}\) are optimal symplectic connections, \[\left.\frac{d}{dt}\right|_{t=0}\N(\psi_t)=\left.\frac{d}{dt}\right|_{t=1}\N(\psi_t)=0.\] By convexity along geodesics, this implies \(\N(\psi_t)=0\) for all \(t\). Hence \(\mathcal{R}_t\dot{\psi}_t=0\) for all \(t\), and in particular \(\mathcal{R}_0\dot{\psi}_0=0\) so that \(v\) is globally holomorphic. Finally, we get \(\omega_X'=\exp(v)^*\omega_X+i\d\db\pi^*f\).
\end{proof}

The following result generalises the boundedness result \ref{thm:DS_boundedness} of Dervan--Sektnan \cite{DS_uniqueness}.

\begin{theorem}[Boundedness]\label{thm:boundedness}
	If \(X\to B\) admits an optimal symplectic connection \(\omega_X'\), then the functional \(\mathcal{N}:\mathcal{K}_E\to\R\) is bounded (with respect to any choice of reference metric \(\omega_X\)).
\end{theorem}

\begin{proof}
	Let \(\varphi:[0,1]\to\K_E\) be a geodesic with \(\omega_X'=\omega_X+i\d\db\varphi_0\). Then \[\left.\frac{d}{dt}\right|_{t=0}\N(\varphi_t)=0,\] and by convexity of \(\N\) along \(\varphi\) we have \[\N(\varphi_1)\geq\mathcal{N}(\varphi_0).\] By geodesic connectedness of \(\K_E\), we have \(\mathcal{N}(\varphi)\geq\mathcal{N}(\varphi_0)\) for all \(\varphi\in\K_E\).
\end{proof}

\section{A Chen--Tian style formula}\label{sec:Chen--Tian}

Recall that for a compact K{\"a}hler manifold \((Y,\omega)\) of dimension \(d\), the Mabuchi functional is defined on the space of K{\"a}hler potentials \(\K\) with respect to \(\omega\), and its derivative along a smooth path \(\varphi_t\) in \(\K\) is \[\frac{d}{dt}\mathcal{M}(\varphi_t):=\int_Y\dot{\varphi}_t(\hat{S}-S(\omega_\varphi))\,\omega_\varphi^d.\] The Mabuchi functional has an explicit formula, which may be derived by taking a straight line path \(\varphi_t:=t\varphi\) from the origin to \(\varphi\): \begin{proposition}[Chen--Tian formula,  {\cite{Che00b,Tia00}}]\label{prop:Chen-Tian}
	For any \(\varphi\in\K\), \begin{align*}
	\mathcal{M}(\varphi)=\int_Y\log\left(\frac{\omega_\varphi^d}{\omega^d}\right)\omega_\varphi^d+&\sum_{j=0}^{d-1}\int_Y\varphi\,\Ric(\omega)\wedge\omega^j\wedge\omega_\varphi^{d-j}\\+&\frac{1}{d+1}\mu(Y)\sum_{j=0}^d\int_Y\varphi\,\omega^j\wedge\omega_\varphi^{d-j},
	\end{align*} where \[\mu(Y):=\frac{d\,c_1(Y)\cdot [\omega]^{d-1}}{[\omega]^d}.\]
\end{proposition} It is often useful to define the functionals appearing in the formula: \begin{align*}
H(\varphi)&:=\int_Y\log\left(\frac{\omega_\varphi^d}{\omega^d}\right)\omega_\varphi^d, \\
R(\varphi)&:=\sum_{j=0}^{d-1}\int_Y\varphi\,\Ric(\omega)\wedge\omega^j\wedge\omega_\varphi^{d-j}, \\
I(\varphi)&:=\frac{1}{d+1}\sum_{j=0}^d\int_Y\varphi\,\omega^j\wedge\omega_\varphi^{d-j}.
\end{align*} The Chen--Tian formula then states \[\mathcal{M}(\varphi)=H(\varphi)+R(\varphi)+\mu(Y)I(\varphi).\]

Now, as usual suppose we have a fibration \(X\to B\) where \(B\) has dimension \(n\) and the fibres \(X_b\) have dimension \(m\). We assume a relatively cscK metric \(\omega_X\) on \(X\) and fix a K{\"a}hler metric \(\omega_B\) on \(B\). The purpose of this section is to prove an analogue of the Chen--Tian formula for the functional \(\mathcal{N}\). This will be achieved through the following result:

\begin{proposition}[{\cite[Lemma 4.26]{DS_uniqueness}}]\label{prop:Mabuchi_expansion}
	For \(k\gg0\), write \(\omega_k:=\omega_X+k\omega_B\), and let \(\mathcal{M}_k\) be the corresponding Mabuchi functional restricted to \(\K_E\). Then there is an expansion \[\mathcal{M}_k=k^n\mathcal{F}+k^{n-1}\mathcal{N}+\O(k^{n-2}),\] where: \begin{enumerate}
	\item \(\mathcal{F}\) is the functional defined by \[\mathcal{F}(\varphi):=\int_B \mathcal{M}_{\V}(\varphi)\,\omega_B^n,\] where \(\mathcal{M}_{\V}\) is the fibrewise Mabuchi functional defined by \(\mathcal{M}_{\V}(\varphi)(b):=\mathcal{M}_{X_b}(\varphi|_{X_b})\), for \(\mathcal{M}_{X_b}\) the Mabuchi functional of the fibre \((X_b,\omega_b)\), and
	\item \(\mathcal{N}\) is the log-norm functional of Proposition \ref{prop:log-norm}.
	\end{enumerate} 
\end{proposition} The proof in \cite{DS_uniqueness} is strictly speaking for the twisted Mabuchi functional, however an entirely analogous proof yields the above result. We also remark that the functional \(\mathcal{F}\) is constant along \(\K_E\), since the Mabuchi functional itself is constant along cscK potentials.

Our goal is now to expand the usual Chen--Tian formula in \(k\) and identify the \(k^{n-1}\)-coefficient to produce a similar formula for the functional \(\mathcal{N}\). Recall the functionals \(H\), \(R\) and \(I\) from the Chen--Tian formula; we will denote by \(H_k\), \(R_k\) and \(I_k\) their counterparts for \(\omega_k\), and expand each of these separately. We remark that the slope \[\mu_k(X):=\frac{(m+n)c_1(X)\cdot[\omega_X+k\omega_B]^{m+n-1}}{[\omega_X+k\omega_B]^{m+n}}\] has an expansion \[\mu_k(X)=A_0(X)+A_1(X)k^{-1}+\O(k^{-2})\] which we will not write again explicitly, but is the same as \eqref{eq:slope_expansion}, with \([\omega_B]\) in place of \(L\) and \([\omega_X]\) in place of \(H\).

\begin{proposition}
	The following expansions hold: \begin{align*}
	I_k(\varphi)=&k^n\frac{1}{m+1}\binom{m+n}{n}\sum_{j=0}^m\int_X\varphi\,\omega_B^n\wedge\omega_X^j\wedge\omega_{X,\varphi}^{m-j} \\
	+& k^{n-1}\frac{1}{m+2}\binom{m+n}{n-1}\sum_{j=0}^{m+1}\int_X\varphi\,\omega_B^{n-1}\wedge\omega_X^j\wedge\omega_{X,\varphi}^{m+1-j} \\
	+&\O(k^{n-2}),
	\end{align*} \begin{align*}
	R_k(\varphi)=&k^n\binom{m+n}{n}\sum_{j=0}^{m-1}\int_X\varphi\,\rho\wedge\omega_B^n\wedge\omega_X^j\wedge\omega_{X,\varphi}^{m-1-j} \\
	+& k^{n-1}\binom{m+n}{n-1}\sum_{j=0}^m\int_X\varphi\,(\rho+\Ric(\omega_B))\wedge\omega_B^{n-1}\wedge\omega_X^{m-j}\wedge\omega_{X,\varphi}^j \\
	+& k^{n-1}\binom{m+n}{n}\sum_{j=0}^{m-1}\int_X\varphi\,i\d\db(\Lambda_{\omega_B}\omega_X)\wedge\omega_B^n\wedge\omega_X^{m-1-j}\wedge\omega_{X,\varphi}^j \\
	+&\O(k^{n-2}),
	\end{align*} 
	and \begin{align*}
	H_k(\varphi)=&k^n\binom{m+n}{n}\int_X\log\left(\frac{\omega_B^n\wedge\omega_{X,\varphi}^m}{\omega_B^n\wedge\omega_X^m}\right)\omega_B^n\wedge\omega_{X,\varphi}^m \\
	+&k^{n-1}\binom{m+n}{n-1}\int_X \log\left(\frac{\omega_B^n\wedge\omega_{X,\varphi}^m}{\omega_B^n\wedge\omega_X^m}\right)\omega_B^{n-1}\wedge\omega_{X,\varphi}^{m+1}  \\
	-&k^{n-1}\binom{m+n}{n}\sum_{j=0}^{m-1}\int_X\varphi\,i\d\db(\Lambda_{\omega_B}\omega_X)\wedge\omega_B^n\wedge\omega_X^{m-1-j}\wedge\omega_{X,\varphi}^j \\
	+&\O(k^{n-2}).
	\end{align*}
\end{proposition}

\begin{proof}
	We will not provide the full proofs for all of these; instead we will content ourselves with calculating the leading order term of \(I_k\) to give an idea of how these calculations can be done for \(I_k\) and \(R_k\), and also calculate the expansion of \(H_k\) since the \(k^{n-1}\) term is rather difficult. For proofs of the remaining expansions, see the author's PhD thesis \cite[Chapter 4.3]{Hal22}.
	
	An important tool in the proofs is the combinatorial formula \begin{equation}\label{eq:comb}
	\binom{d+1}{p+1}=\binom{d}{p}+\binom{d-1}{p}+\cdots+\binom{p}{p}.
	\end{equation} This, for example, lets one write the functional \(I\) on a manifold \((Y,\omega)\) of dimension \(d\) instead as \[I(\varphi)=\frac{1}{d+1}\sum_{p=0}^d\binom{d+1}{p+1}\int_Y\varphi\,\omega^{d-p}\wedge(i\d\db\varphi)^p.\] This is achieved by expanding \(\omega_\varphi^j=(\omega+i\d\db\varphi)^j\), and reindexing the sum over powers of \(i\d\db\varphi\): \begin{align*}
	I(\varphi)&=\frac{1}{d+1}\sum_{j=0}^d\int_Y\varphi\,\omega_\varphi^j\wedge\omega^{d-j} \\
	&=\frac{1}{d+1}\sum_{j=0}^d\int_Y\varphi\sum_{\ell=0}^j\binom{j}{\ell}(i\d\db\varphi)^\ell\wedge\omega^{d-\ell} \\
	&= \frac{1}{d+1}\sum_{p=0}^d\left(\binom{d}{p}+\binom{d-1}{p}+\cdots+\binom{p}{p}\right)\int_Y\varphi(i\d\db\varphi)^p\wedge\omega^{d-p}\\
	&=\frac{1}{d+1}\sum_{p=0}^d\binom{d+1}{p+1}\int_Y\varphi\,\omega^{d-p}\wedge(i\d\db\varphi)^p.
	\end{align*} Conversely, if one obtains a formula of the form in the final line, then one can factorise it using exactly the same process.
	
	First, for the expansion of \(I_k\) we have \[I_k(\varphi)=\frac{1}{m+n+1}\sum_{p=0}^{m+n}\binom{m+n+1}{p+1}\int_X\varphi\,\omega_k^{m+n-p}\wedge(i\d\db\varphi)^{p}.\] Expanding this expression in \(k\), we find the leading coefficient of \(k^n\) is \[\frac{1}{m+n+1}\sum_{p=0}^{m}\binom{m+n+1}{p+1}\int_X\varphi\binom{m+n-p}{n}\omega_B^n\wedge\omega_X^{m-p}\wedge(i\d\db\varphi)^{p}.\] Now note that \begin{align*}
	\frac{1}{m+n+1}\binom{m+n+1}{p+1}\binom{m+n-p}{n} &= \frac{(m+n)!}{(p+1)!(m+n-p)!}\,\frac{(m+n-p)!}{n!(m-p)!} \\
	&=\frac{(m+n)!}{(m+1)!\,n!}\,\frac{(m+1)!}{(p+1)!(m-p)!} \\
	&= \frac{1}{m+1}\binom{m+n}{n}\binom{m+1}{p+1}.
	\end{align*} Hence the leading coefficient of \(I_k(\varphi)\) is \[\frac{1}{m+1}\binom{m+n}{n}\sum_{p=0}^m\binom{m+1}{p+1}\int_X\varphi\,\omega_B^n\wedge\omega_X^{m-p}\wedge(i\d\db\varphi)^p\] which by the combinatorial formula \eqref{eq:comb} can be factorised as \[\frac{1}{m+1}\binom{m+n}{n}\sum_{j=0}^m\int_X\varphi\,\omega_B^n\wedge\omega_X^j\wedge\omega_{X,\varphi}^{m-j}.\]
	
	The calculation for the subleading order term is similar. For the \(R\)-functional, one must also use the expansion of the Ricci curvature given by \[\Ric(\omega_k)=\rho+\Ric(\omega_B)+k^{-1}i\d\db(\Lambda_{\omega_B}\omega_X)+\O(k^{-2});\] see \cite[Lemma 4.6]{DS_osc}.
	
	Finally we expand the \(H\)-functional, \[H_k(\varphi)=\int_X\log\left(\frac{\omega_{k,\varphi}^{m+n}}{\omega_k^{m+n}}\right)\omega_{k,\varphi}^{m+n}.\] We first expand the logarithm. Towards this, we have the expansion \begin{align*}
	\omega_{k}^{m+n} &=k^n\binom{m+n}{n}\omega_B^n\wedge\omega_X^m+k^{n-1}\binom{m+n}{n-1}\omega_B^{n-1}\wedge\omega_X^{m+1}+\O(k^{n-2})\\
	&=k^n\binom{m+n}{n}\omega_B^n\wedge\omega_X^m\left(1+k^{-1}\frac{n}{m+1}\frac{\omega_B^{n-1}\wedge\omega_X^{m+1}}{\omega_B^n\wedge\omega_X^m}+\O(k^{-2})\right).
	\end{align*} Applying this same expansion to \(\omega_{k,\varphi}^{m+n}\) gives \begin{align*}
	\frac{\omega_{k,\varphi}^{m+n}}{\omega_{k}^{m+n}} &= \frac{\omega_B^n\wedge\omega_{X,\varphi}^m}{\omega_B^n\wedge\omega_X^m}\left(1+k^{-1}\frac{n}{m+1}\left(\frac{\omega_B^{n-1}\wedge\omega_{X,\varphi}^{m+1}}{\omega_B^n\wedge\omega_{X,\varphi}^m}-\frac{\omega_B^{n-1}\wedge\omega_X^{m+1}}{\omega_B^n\wedge\omega_X^m}\right)+\O(k^{-2})\right)
	\end{align*} where we have used \[\frac{1+k^{-1}a+\O(k^{-2})}{1+k^{-1}b+\O(k^{-2})}=1+k^{-1}(a-b)+\O(k^{-2})\] for real numbers \(a,b\). Applying the logarithm to \(\omega_{k,\varphi}^{m+n}/\omega_k^{m+n}\) then gives \[\log\left(\frac{\omega_B^n\wedge\omega_{X,\varphi}^m}{\omega_B^n\wedge\omega_X^m}\right)+\log\left(1+k^{-1}\frac{n}{m+1}\left(\frac{\omega_B^{n-1}\wedge\omega_{X,\varphi}^{m+1}}{\omega_B^n\wedge\omega_{X,\varphi}^m}-\frac{\omega_B^{n-1}\wedge\omega_{X}^{m+1}}{\omega_B^n\wedge\omega_{X}^m}\right)+\O(k^{-2})\right).\] Using the power series formula \[\log(1+x)=\sum_{i=1}^\infty(-1)^{i+1}\frac{x^i}{i}\] we get \[\log\left(\frac{\omega_{k,\varphi}^{m+n}}{\omega_k^{m+n}}\right)=\log\left(\frac{\omega_B^n\wedge\omega_{X,\varphi}^m}{\omega_B^n\wedge\omega_X^m}\right)+k^{-1}\frac{n}{m+1}\left(\frac{\omega_B^{n-1}\wedge\omega_{X,\varphi}^{m+1}}{\omega_B^n\wedge\omega_{X,\varphi}^m}-\frac{\omega_B^{n-1}\wedge\omega_{X}^{m+1}}{\omega_B^n\wedge\omega_{X}^m}\right)+\O(k^{-2}).\] The \(k^n\)-coefficient of \(H_k(\varphi)\) is therefore \[\binom{m+n}{n}\int_X\log\left(\frac{\omega_B^n\wedge\omega_{X,\varphi}^m}{\omega_B^n\wedge\omega_X^m}\right)\omega_B^n\wedge\omega_{X,\varphi}^m.\] The \(k^{n-1}\)-coefficient is \begin{align}\label{eq:dummy1}
	&\nonumber\binom{m+n}{n-1}\int_X\left(\frac{\omega_B^{n-1}\wedge\omega_{X,\varphi}^{m+1}}{\omega_B^n\wedge\omega_{X,\varphi}^m}-\frac{\omega_B^{n-1}\wedge\omega_{X}^{m+1}}{\omega_B^n\wedge\omega_{X}^m}\right)\omega_B^n\wedge\omega_{X,\varphi}^m \\
	&+\binom{m+n}{n-1}\int_X \log\left(\frac{\omega_B^n\wedge\omega_{X,\varphi}^m}{\omega_B^n\wedge\omega_X^m}\right)\omega_B^{n-1}\wedge\omega_{X,\varphi}^{m+1}.
	\end{align} We can reduce the first term of \eqref{eq:dummy1} as follows: \begin{align*}
	&\int_X\left(\frac{\omega_B^{n-1}\wedge\omega_{X,\varphi}^{m+1}}{\omega_B^n\wedge\omega_{X,\varphi}^m}-\frac{\omega_B^{n-1}\wedge\omega_{X}^{m+1}}{\omega_B^n\wedge\omega_{X}^m}\right)\omega_B^n\wedge\omega_{X,\varphi}^m \\
	=& \int_X\left(\omega_B^{n-1}\wedge\omega_{X,\varphi}^{m+1}-\frac{\omega_B^n\wedge\omega_{X,\varphi}^m}{\omega_B^n\wedge\omega_X^m}\,\omega_B^{n-1}\wedge\omega_X^{m+1}\right)  \\
	=&
	\int_X\left(1-\frac{\omega_B^n\wedge\omega_{X,\varphi}^m}{\omega_B^n\wedge\omega_X^m}\right)\omega_B^{n-1}\wedge\omega_X^{m+1} \\
	=&
	\int_X\frac{\omega_B^n\wedge\omega_X^m-\omega_B^n\wedge\omega_{X,\varphi}^m}{\omega_B^n\wedge\omega_X^m}\,\omega_B^{n-1}\wedge\omega_X^{m+1} \\
	=&
		\int_X\frac{(m+1)\omega_B^{n-1}\wedge\omega_{X,\mathcal{H}}\wedge\omega_{X,\mathcal{V}}^m}{\omega_B^n\wedge\omega_{X,\V}^m}(\omega_B^n\wedge\omega_X^m-\omega_B^n\wedge\omega_{X,\varphi}^m) \\
	=&
	\frac{m+1}{n}\int_X(\Lambda_{\omega_B}\omega_X)(\omega_B^n\wedge\omega_X^m-\omega_B^n\wedge\omega_{X,\varphi}^m) \\
	=&
	-\frac{m+1}{n}\sum_{p=1}^m\binom{m}{p}\int_X(\Lambda_{\omega_B}\omega_X)\omega_B^n\wedge\omega_X^{m-p}\wedge(i\d\db\varphi)^p \\
	=&
	-\frac{m+1}{n}\sum_{p=1}^m\binom{m}{p}\int_X\varphi\,i\d\db(\Lambda_{\omega_B}\omega_X)\wedge\omega_B^n\wedge\omega_X^{m-p}\wedge(i\d\db\varphi)^{p-1} \\
	=&
	-\frac{m+1}{n}\sum_{p=0}^{m-1}\binom{m}{p+1}\int_X\varphi\,i\d\db(\Lambda_{\omega_B}\omega_X)\wedge\omega_B^n\wedge\omega_X^{m-1-p}\wedge(i\d\db\varphi)^p \\
	=&
	-\frac{m+1}{n}\sum_{j=0}^{m-1}\int_X\varphi\,i\d\db(\Lambda_{\omega_B}\omega_X)\wedge\omega_B^n\wedge\omega_X^{m-1-j}\wedge\omega_{X,\varphi}^j,
	\end{align*} where we used \[\omega_B^{n-1}\wedge\omega_X^{m+1}=\omega_B^{n-1}\wedge(\omega_{X,\mathcal{H}}+\omega_{X,\mathcal{V}})^{m+1}=(m+1)\omega_B^{n-1}\wedge\omega_{X,\mathcal{H}}\wedge\omega_{X,\V}^m\] in the fifth line, and in the final line we have again used the identity \eqref{eq:comb}.
\end{proof}

Combining all of these expansions together with Proposition \ref{prop:Mabuchi_expansion}, we get the following. Note there is a miraculous cancellation that occurs between the \(k^{n-1}\)-coefficients of \(H_k\) and \(R_k\), following from the lengthy calculation at the end of the last proof.

\begin{theorem}\label{prop:Chen-Tian_fib}
	The functional \(\N\) satisfies the formula \begin{align*}
	\N(\varphi)  = \binom{m+n}{n-1}&\left[\, \int_X\log\left(\frac{\omega_B^n\wedge\omega_{X,\varphi}^m}{\omega_B^n\wedge\omega_X^m}\right)\,\omega_B^{n-1}\wedge\omega_{X,\varphi}^{m+1} \right. \\
	&+\sum_{j=0}^m\int_X\varphi\,(\rho+\Ric(\omega_B))\wedge\omega_B^{n-1}\wedge\omega_X^{m-j}\wedge\omega_{X,\varphi}^j \\
	&+\frac{1}{m+2}A_0(X)\sum_{j=0}^{m+1}\int_X\varphi\,\omega_B^{n-1}\wedge\omega_X^j\wedge\omega_{X,\varphi}^{m+1-j}\\
	&\left.+\frac{1}{n}A_1(X)\sum_{j=0}^m\int_X\varphi\,\omega_B^n\wedge\omega_X^j\wedge\omega_{X,\varphi}^{m-j}\,\right],
	\end{align*} where \(\rho\) is the curvature of the metric on \(\Lambda^m\V\) determined by \(\omega_X\), and \(A_0(X)\) and \(A_1(X)\) are the topological constants appearing in the expansion \[\mu_k(X)=A_0(X)+k^{-1}A_1(X)+\O(k^{-2}).\]
\end{theorem}

As with the standard Chen--Tian formula, it will be useful to name the functionals appearing in the formula for later use: \begin{align}
\widetilde{H}(\varphi):=&\int_X\log\left(\frac{\omega_B^n\wedge\omega_{X,\varphi}^m}{\omega_B^n\wedge\omega_X^m}\right)\,\omega_B^{n-1}\wedge\omega_{X,\varphi}^{m+1} \nonumber\\
\widetilde{R}(\varphi):=&\sum_{j=0}^m\int_X\varphi\,(\rho+\Ric(\omega_B))\wedge\omega_B^{n-1}\wedge\omega_X^{m-j}\wedge\omega_{X,\varphi}^j \nonumber\\
\widetilde{I}(\varphi):=&\sum_{j=0}^m\int_X\varphi\,\omega_B^n\wedge\omega_X^j\wedge\omega_{X,\varphi}^{m-j} \nonumber \\
\widetilde{J}(\varphi):=&\sum_{j=0}^{m+1}\int_X\varphi\,\omega_B^{n-1}\wedge\omega_X^j\wedge\omega_{X,\varphi}^{m+1-j} \label{eq:fibration_functionals}
\end{align} The formula then takes the more notationally compact form: \[\N(\varphi)=\binom{m+n}{n-1}\left(\widetilde{H}(\varphi)+\widetilde{R}(\varphi)+\frac{1}{n}A_1(X)\widetilde{I}(\varphi)+\frac{1}{m+2}A_0(X)\widetilde{J}(\varphi)\right).\]

\section{Geodesics associated to fibration degenerations and polystability}\label{polystability}

In this section, we associate to certain fibration degenerations \((\X,\H)\to(B,L)\) whose fibres over \(B\) are all product test configurations, a unique geodesic ray in \(\K_E\). Calculating the limiting slope of the functional \(\N\) along this ray, we recover the invariant \(W_1(\X,\H)\) determining stability of the fibration. Applying convexity of the functional \(\N\) along geodesics, this yields polystability of fibrations admitting optimal symplectic connections, with respect to fibration degenerations that are locally products over \(B\times\C\). For isotrivial fibrations, meaning those whose fibres are all isomorphic as complex manifolds, this implies polystability with respect to fibration degenerations that are smooth over \(B\times\C\). In the case of varieties, the ideas used here go back to \cite{Tia97}, and have been carried out by various authors in \cite{PS10}, \cite{PRS08}, \cite{BHJ19}, with extensions to the non-algebraic K{\"a}hler case in \cite{DR17} and \cite{Sjo18,Sjo20}.

These results will apply under the following assumptions: we consider a smooth projective morphism \(\pi:X\to B\) of non-singular projective varieties, together with a choice of ample line bundle \(L\to B\) and relatively ample line bundle \(H\to X\). We further assume that \(\dim H^0(X_b,TX_b)\) is independent of \(b\), which by the partial Cartan decomposition is equivalent to assuming that \(\dim E_b\) is independent of \(b\), and assume that each fibre \(X_b\) admits a cscK metric in the class \(c_1(H_b)\). Even though we have chosen to work in the algebraic setting with line bundles, the results will carry over easily to the transcendental K{\"a}hler setting, as we comment on in Remark \ref{rem:Kahler_case}.

\subsection{Deligne Pairings}

Towards calculating the limiting slope of the functional \(\N\), we first review the theory of Deligne pairings; see \cite[Section 1.2]{BHJ19} for further information. Let \(Y\) and \(T\) be smooth varieties, and let \(\pi:Y\to T\) be a flat projective morphism of relative dimension \(m\). Given line bundles \(L_1,\ldots,L_{m+1}\) on \(Y\), integration over the fibres gives a well-defined cohomology class \[\int_{Y/T}c_1(L_1)\cdots c_1(L_{m+1})\] on the base \(T\). The theory of Deligne pairings gives a canonical choice of line bundle \(\langle L_1,\ldots,L_{m+1}\rangle_{Y/T}\) whose first Chern class is this cohomology class. Moreover, if \(h_1,\ldots,h_{m+1}\) are smooth hermitian metrics on \(L_1,\ldots,L_{m+1}\) respectively, then the theory also produces a canonical \emph{continuous} hermitian metric \(\langle h_1,\ldots,h_{m+1}\rangle_{Y/T}\) on \(\langle L_1,\ldots,L_{m+1}\rangle_{Y/T}\) that is smooth over the smooth locus of \(\pi\). The Deligne pairings are symmetric in the factors, multilinear, and functorial. Most importantly, they satisfy: \begin{enumerate}
	\item The curvature of \(\langle h_1,\ldots,h_{m+1}\rangle_{Y/T}\) is given by \[\int_{Y/T}\omega_1\wedge\cdots\wedge\omega_{m+1},\] where \(\omega_j\) is the curvature of \(h_j\).
	\item If \(h_1'\) is another smooth hermitian metric on \(L_1\), then \begin{equation}\label{eq:change_of_metric}
	\log\left(\frac{\langle h_1',h_2,\ldots,h_{m+1}\rangle_{Y/T}}{\langle h_1,h_2,\ldots,h_{m+1}\rangle_{Y/T}}\right)=\int_{Y/T}\phi\,\omega_2\wedge\cdots\wedge\omega_{m+1},
	\end{equation} where \(\phi:=\log(h_1'/h_1)\).
\end{enumerate}

To simplify notation, we shall abusively write \begin{equation}\label{eq:abuse}
\langle h_1',h_2,\ldots,h_{m+1}\rangle_{Y/T}-\langle h_1,h_2,\ldots,h_{m+1}\rangle_{Y/T}
\end{equation} for what is really \[\log\left(\frac{\langle h_1',h_2,\ldots,h_{m+1}\rangle_{Y/T}}{\langle h_1,h_2,\ldots,h_{m+1}\rangle_{Y/T}}\right).\] This is in line with using additive notation for line bundles and metrics rather than multiplicative notation. Furthermore, when \(T\) is a point we shall simply write \(\langle\, ,\rangle_Y\) in place of \(\langle \, ,\rangle_{Y/\pt}\).

Using the change of metric formula \eqref{eq:change_of_metric}, we can rewrite the functionals appearing in the Chen--Tian style formula \ref{prop:Chen-Tian_fib} in terms of Deligne pairings. Let \(\pi:X\to B\) be a smooth projective morphism of projective varieties, let \(L\to B\) be ample and \(H\to X\) be relatively ample, let \(h_X\) be a smooth metric on \(H\) with curvature a relatively K{\"a}hler form \(\omega_X\), and let \(h_B\) be a smooth metric on \(L\) with curvature a K{\"a}hler form \(\omega_B\). Denote by \(\eta_X\) the metric on \(-K_X\) corresponding to the volume form \(\omega_X^m\wedge\omega_B^n\). Lastly, write \(f=\Lambda_{\omega_B}\omega_X\) which we consider as a hermitian metric on the trivial line bundle on \(X\) with curvature \(i\d\db(\Lambda_{\omega_B}\omega_X)\).

\begin{proposition}\label{prop:Deligne_pairings}
	Let \(h_{X,\varphi}\) be another smooth metric on \(H\) with relatively K{\"a}hler curvature \(\omega_{X,\varphi}\), where \(\varphi=\log(h_{X,\varphi}/h)\), and let \(\eta_{X,\varphi}\) be the metric on \(-K_X\) corresponding to the volume form \(\omega_{X,\varphi}^m\wedge\omega_B^n\). Then: \begin{enumerate}[label=(\roman*)]
		\item \(\widetilde{H}(\varphi)
		= \langle \eta_{X,\varphi},h_B^{n-1},h_{X,\varphi}^{m+1}\rangle_{X}-\langle \eta_X,h_B^{n-1},h_{X,\varphi}^{m+1}\rangle_{X}\),
		\item \(\widetilde{R}(\varphi) = \langle \eta_X,h_B^{n-1},h_{X,\varphi}^{m+1}\rangle_{X}-\langle \eta_X,h_B^{n-1},h_{X}^{m+1}\rangle_{X}\),
		\item \(\widetilde{I}(\varphi) =
		\langle h_B^{n},h_{X,\varphi}^{m+1}\rangle_{X}-\langle h_B^{n},h_{X}^{m+1}\rangle_{X}\),
		\item \(\widetilde{J}(\varphi) =
		\langle h_B^{n-1},h_{X,\varphi}^{m+2}\rangle_{X}-\langle h_B^{n-1},h_{X}^{m+2}\rangle_{X}\),
	\end{enumerate} where the functionals \(\widetilde{H},\widetilde{R},\widetilde{I},\widetilde{J}\) are defined in \eqref{eq:fibration_functionals}, and we abuse notation as in \eqref{eq:abuse}.
\end{proposition}

\begin{proof}
	This is immediate from the change of metric formula \eqref{eq:change_of_metric} for Deligne pairings; we only note that \[\log(\eta_{X,\varphi}/\eta_X)=\log\left(\frac{\omega_{X,\varphi}^m\wedge\omega_B^n}{\omega_X^m\wedge\omega_B^n}\right),\] which is the integrand appearing in the functional \(\widetilde{H}\).
\end{proof}

\subsection{Geodesics associated to fibration degenerations and the Wess--Zumino--Witten equation}

Now, let us turn to fibration degenerations. First, notice that any fibration degeneration with \(W_0(\X,\H)>0\) may be ignored, as these by definition do not destabilise the fibration. We may therefore assume \(W_0(\X,\H)=0\), which by Proposition \ref{prop:flat_locus} means that the general fibre \((\X_b,\H_b)\) normalises to a product test configuration for \((X_b,H_b)\). If we normalise \((\X,\H)\), then the general fibre \((\X_b,\H_b)\) of \((\X,\H)\) over \(B\) will also be normal, and hence will be a genuine product test configuration for \((X_b,H_b)\) by K-polystability of the fibres.

It follows that it suffices to consider fibration degenerations whose general fibre is a genuine product test configuration. In what follows, we will make a slightly more restrictive assumption that implies all fibres are product test configurations: \begin{definition}
	A fibration degeneration \((\X,\H)\) of \((X,H)\) is called \emph{product-type} if for every \(b\in B\) there is an open neighbourhood \(U\) of \(b\) and a \(\C^*\)-equivariant holomorphic isomorphism \[(\X|_{U\times\C},\H|_{U\times\C})\cong (X|_U\times\C,p^*H|_U),\] where \(p:X|_U\times\C\to X|_U\) is the projection, and the \(\C^*\)-action on the right-hand side is the standard multiplication on \(\C\), but may be non-trivial on \((X|_U,H|_U)\). 
\end{definition}

Restricting attention to product-type fibration degenerations is analogous to considering sub-vector bundles instead of subsheaves when studying stability of vector bundles; see Examples \ref{ex:projectivised_vector_bundle} and \ref{ex:principal_bundle_fibration}. Thus, when the base \(B\) of the fibration is a curve, we expect such fibration degenerations to play an important role, perhaps determining stability completely.

A product-type fibration degenerations has all fibres \(\X_b\) product test configurations. When the fibration is isotrivial, the converse holds: \begin{lemma}\label{lem:a}
	Suppose that all fibres \(X_b\) are isomorphic as complex manifolds, and let \(\X\) be a fibration degeneration of \(X\). If all fibres \(\X_b\) are product test configurations, then \(\X\) is product-type.
\end{lemma}

\begin{proof}
	This is a consequence of the Fischer--Grauert theorem \cite{FG65}, which states that if \(Y\to T\) is a proper holomorphic surjective submersion with all fibres \(Y_t\) isomorphic as complex manifolds, then the fibration is holomorphically locally trivial. Applying this to \(\X\) in a neighbourhood of \((b,0)\in U\times\C\) and applying \(\C^*\)-equivariance gives that \(\X|_U\cong X|_U\times\C\). 
	
	It remains to show the isomorphism of line bundles \(\H_{U\times\C}\cong p^*H|_U\). This kind of result is well established, see for example \cite[Lemma 5.10]{New78} which applies in the algebraic case but extends equally well to the holomorphic setting.
\end{proof}

We also have the following, which tells us that for isotrivial fibrations, our result applies to all fibration degenerations that are smooth over \(B\times\C\).

\begin{lemma}\label{lem:b}
	Let \((\X,\H)\) be a fibration degeneration of \((X,H)\) that is smooth over \(B\times\C\), with \(W_0(\X,\H)=0\). Then all fibres \((\X_b,\H_b)\) are product test configurations.
\end{lemma}

Here smoothness is in the sense of scheme theory, so that \(\X\to B\times\C\) is a holomorphic submersion.

\begin{proof}
	Since the map \(\X\to B\times\C\) is smooth, all fibres \(\X_b\) will be normal varieties. Since \(W_0(\X,\H)=0\), by Proposition \ref{prop:flat_locus} the general fibre \((\X_b,\H_b)\) satisfies \(\DF(\X_b,\H_b)=0\). Now, since the morphism \(\X\to B\times\C\) is smooth it is in particular flat. The Donaldson--Futaki invariant is constant in flat families, and therefore every fibre satisfies \(\DF(\X_b,\H_b)=0\). Since every fibre is also normal, it follows that all fibres \((\X_b,\H_b)\) are product test configurations, by K-polystability of the fibres \((X_b,H_b)\).
\end{proof}

Our aim is now to associate to any product-type fibration degeneration a unique geodesic ray \(\varphi:[0,\infty)\to\K_E\) that is compatible with this fibration degeneration in a suitable sense. Firstly, let \(D\subset\C\) be the closed unit disc, and write \(D^\times:=D\backslash\{0\}\). Given a fibration degeneration \((\X,\H)\), we will denote by \(\X|_{D}\) the restriction of \(\X\) to \(D\subset\C\), and similarly for \(\H|_D\). Using the \(\C^*\)-action on \(\X\), we consider \((X\times D^\times,p^*H)\subset(\X|_{D},\H|_D)\) as an open subset, again writing \(p:X\times D^\times\to X\) for the projection.

\begin{definition}
	Let \((\X,\H)\) be a product-type fibration degeneration of \((X,H)\), and let \(h_X\) be a fixed hermitian metric on \(H\) with curvature \(\omega_X\) a relatively cscK metric on \(X\). A smooth geodesic ray \(\varphi:[0,\infty)\to\K_E\) is said to be \emph{compatible} with \((\X,\H)\) if there exists a smooth \(S^1\)-invariant hermitian metric \(h_\X\) on \(\H|_D\) such that the function \(\Phi:X\times D^\times\to\R\) defined by \[h_{\X}|_{D^\times}=e^\Phi p^*h_X\] satisfies: \begin{enumerate}[label=(\roman*)]
		\item \(\Phi|_{\d\X}=0\), and 
		\item \(\Phi(x,\tau):=\varphi_{-\log|\tau|}(x)\) for all \(\tau\in D^\times\).
	\end{enumerate}
\end{definition}

\begin{remark}In the non-relative case, there is already a suitable notion of a geodesic compatible with a given test configuration, see for example \cite{PS10}. Since such a geodesic solves a degenerate Monge-Amp{\`e}re equation with boundary condition, the geodesic is necessarily unique. For arbitrary test configurations, one must first resolve singularities on the total space, and even then the geodesic is not guaranteed to be smooth but only \(C^{1,1}\) in general. When the test configuration is a product however, we shall see that the unique associated geodesic ray is smooth. Moreover if the initial condition \(h_X\) is a cscK metric, then the geodesic ray will also be through cscK metrics. 
\end{remark}

We wish to associate a unique smooth geodesic ray to any product-type fibration degeneration. As for the geodesic equation in Section \ref{sec:geodesic_equation}, the notion of geodesic ray compatible with a fibration degeneration is simply the usual notion of compatibility with test configurations but fibrewise over \(B\). So we will again construct the geodesics fibrewise and show they vary smoothly with the base point.

First, the fibrewise geodesics must be shown to exist, be smooth, and be through
cscK potentials. This is the content of the following proposition, which is a basic
consequence of results by Donaldson (\cite[pp. 100-101]{Don92}, \cite[Propositions 3, 4]{Don99}) as explained below.

\begin{proposition}
	Let \((X\times\C,p^*L)_\rho\) be a product test configuration for \((X,L)\), and let \(\omega\) be a cscK metric in \(c_1(L)\). Then there exists a unique geodesic ray compatible with the product test configuration having initial condition \(0\in\mathcal{K}_\omega\).
\end{proposition}

\begin{proof}
This essentially follows from Donaldson's reduction of the geodesic equation on the disc to a so-called ``Wess--Zumino--Witten" (WZW) equation. Denote by \(D\) the disc and let \(G/K\) be the symmetric space \(\Aut_{\red}(X)/(\Isom_0(\omega)\cap\Aut_{\red}(X))\), where \(\omega\) is cscK. Then \(G/K\) is a non-positively curved symmetric space, embedded in the space \(\H\) of K{\"a}hler potentials as a totally geodesic submanifold. We wish to construct a smooth map \(f:D\to G/K\) satisfying the geodesic equation together with the boundary condition given by \(f(e^{i\theta})=K\cdot\rho(e^{i\theta})\), where \(\rho:\C^*\to\Aut(X,L)\) is the action determining the product test configuration. 

By \cite[Propositions 3, 4]{Don99}, \(f:D\to G/K\) defines a geodesic if and only if \(f\) solves the WZW equation \begin{equation}\label{eq:WZW}
\tilde{\tau}(f):=\tau(f)+[D_uf,D_vf]=0.
\end{equation} Here \(\tau(f):=\tr(\nabla df)\) is the usual tension field appearing in the harmonic map equation, and \([\,,]:T(G/K)\otimes T(G/K)\to T(G/K)\) is the Lie bracket given by the canonical isomorphism of the tangent space with the Lie algebra of \(K\). Also, \(u+iv\) is the standard complex coordinate on \(D\subset\C\), and \(D_uf:=df(\d_u)\), \(D_vf:=df(\d_v)\). As is proven in \cite[pp. 100-101]{Don92} (see also the comments in Section 7 of \cite{Don99}), the WZW equation has a unique smooth solution for any fixed boundary condition. It follows that the geodesic equation on a product test configuration is uniquely soluble, and the solution is smooth and lies within the space of cscK metrics.
\end{proof}

Hence, on a product-type fibration degeneration, each fibre \((\X_b,\H_b)\) has a unique geodesic ray \(\varphi_b:[0,\infty)\to\K_{E,b}\) associated to it. It once again falls to proving these geodesic rays depend smoothly on the base point \(b\). This will be done by linearising the WZW equation at a solution, and showing the derivative is an isomorphism. This has been carried out by Goldstein \cite{Gol72} for the standard harmonic map equation; adapting his proof, we show:

\begin{proposition}\label{thm:WZW_linearisation}
	For \(r\geq2\) and \(\alpha\in[0,1]\), let \[F:C^{r,\alpha}(D,G/K)\to C^{r-2,\alpha}(D,T(G/K))\times C^{r,\alpha}(\d D,G/K)\] be given by \[F(f)=(\tilde{\tau}(f),f|_{\d D}),\] where \(\tilde{\tau}(f):=\tr(\nabla df)+[D_uf,D_vf]\) is the WZW operator \eqref{eq:WZW}. If \(f:D\to G/K\) is a solution to the WZW equation, then \(dF|_f\) is an isomorphism.
\end{proposition}

\begin{proof}
	Following \cite{Gol72}, we show the derivative is both injective and surjective at a solution. In fact, surjectivity follows from abstract index theory, given that the leading order term in the linearisation of the WZW equation agrees with that of the harmonic map equation; see \cite[pp. 353--355]{Gol72}. Hence we only need to show that the derivative is injective at a solution.
	
	So, suppose that \(W\in T_fC^{r,\alpha}(D,G/K)=C^{r,\alpha}(D,f^*T(G/K))\) lies in the kernel of \(dF|_f\). Then \(W|_{\d D}=0\), and \(W\) arises from a variation \(f_t\) of \(f\) with \(f_0=f\) and \(f_t|_{\d D}=f|_{\d D}\) for all \(t\). Let \(u+iv\) be the standard complex coordinate on \(\C\) restricted to the closed unit disc \(D\subset\C\). In these coordinates, the Wess--Zumino--Witten equation takes the form \[\tilde{\tau}(f)=D_uD_uf+D_vD_vf+[D_uf,D_vf]=0,\] where \(D_uf:=f_*(\d_u)\), \(D_vf:=f_*(\d_v)\), and for a vector field \(V\) along \(f\), \(D_uV\) (resp. \(D_vV\)) is the covariant derivative of \(V\) along \(f_*(\d_u)\) (resp. \(f_*(\d_v)\)). In \cite[``Lemma"]{Gol72}, Goldstein shows that that \(W\) lying in the kernel of \(dF|_f\) is equivalent to the condition\footnote{It is here that one uses the condition that \(f\) solves the WZW equation.} \[0=D_t(\tau(f_t)+[D_uf_t,D_vf_t])|_{t=0}.\] Omitting the evaluation at \(t=0\), the WZW equation then differentiates to give \begin{align*}
	0=&D_tD_uD_uf+D_tD_vD_vf+D_t[D_uf,D_vf] \\
	=&D_uD_tD_uf+R(D_tf,D_uf)D_uf \\
	+&D_vD_tD_vf+R(D_tf,D_vf)D_vf+D_t[D_uf,D_vf],
	\end{align*} where \(R\) is the curvature tensor of \(G/K\). Next, we use that \([\,,]\) is covariantly constant on \(G/K\), so that \(D_t[D_uf,D_vf]=[D_tD_uf,D_vf]+[D_uf,D_tD_vf]\). Also, the connection on \(G/K\) is torsion-free, so we have \(D_tD_uf=D_uD_tf\). Lastly, the curvature operator \(R\) on \(G/K\) takes the explicit form \(R(U,V)Z=-\frac{1}{4}[Z,[U,V]]\). Putting this all together, we get \begin{align*}
	0=& D_uD_uW+D_vD_vW-\frac{1}{4}[D_uf,[W,D_uf]]-\frac{1}{4}[D_vf,[W,D_vf]] \\
	+&[D_uW,D_vf]+[D_uf,D_vW].
	\end{align*} Taking the \(L^2\)-inner product with \(-W\), \begin{align*}
	0=&-\langle W,D_uD_uW\rangle-\langle W,D_vD_vW\rangle\\
	&+\frac{1}{4}\langle W,[D_uf,[W,D_uf]]\rangle+\frac{1}{4}\langle W,[D_vf,[W,D_vf]]\rangle \\
	&-\langle W,[D_uW,D_vf]\rangle-\langle W,[D_uf,D_vW]\rangle.
	\end{align*} Since \(W|_{\d D}=0\), we can integrate \(D_u\) and \(D_v\) by parts without picking up boundary terms, hence \(-\langle W,D_uD_uW\rangle=\langle D_uW,D_uW\rangle\) and similarly for \(v\). Furthermore, the inner product on \(G/K\) is ad-invariant, so that \(\langle [U,V],W\rangle+\langle V,[U,W]\rangle=0\). Hence we get \begin{align*}
	0=&\langle D_uW,D_uW\rangle+\langle D_vW,D_vW\rangle\\
	+&\frac{1}{4}\langle [D_uf,W],[D_uf,W]\rangle+\frac{1}{4}\langle [D_vf,W],[D_vf,W]\rangle \\
	+&\langle [D_uf,W],D_vW\rangle-\langle [D_vf,W],D_uW\rangle\\
	=&\|D_uW-\frac{1}{2}[D_vf,W]\|^2+\|D_vW+\frac{1}{2}[D_uf,W]\|^2
	\end{align*} Since all terms in the final line are non-negative, they must vanish. Hence \(D_uW=\frac{1}{2}[D_vf,W]\) and \(D_vW=-\frac{1}{2}[D_uf,W]\). Now, letting \((\,,)\) denote the pointwise inner product on \(T(G/K)\), we have \[\d_u( W,W)=2(D_uW,W)=([D_vf,W],W)=(D_vf,[W,W])=0,\] and similarly \(\d_v(W,W)=0\). It follows that \((W,W)\) is constant, and since \(W|_{\d D}=0\) we have \(W=0\).
\end{proof}

\begin{theorem}
	Let \((\X,\H)\) be a product-type fibration degeneration of \((X,H)\). Given any relatively cscK metric \(\omega_X\) on \(X\), there exists a unique smooth geodesic ray \(\varphi:[0,\infty)\to\K_E\) compatible with \((\X,\H)\).
\end{theorem}

\begin{proof}
	Over an open subset \(U\subset B\), we trivialise the fibration degeneration, \[(\X_{U\times\C},\H|_{U\times\C})\cong(X|_U\times\C,p^*H).\] On \(U\), construct the fibre bundles: \begin{enumerate}
		\item \(\mathcal{D}:=U\times D\),
		\item \(\d\mathcal{D}:=U\times\d D\subset\mathcal{D}\),
		\item \(\mathcal{G}/\mathcal{K}\) the smooth fibre bundle whose fibre over \(b\in U\) is the homogeneous space \(G_b/K_b\) of cscK metrics in \(c_1(H_b)\), as in Section \ref{sec:geodesic_equation}, and
		\item \(T(\mathcal{G}/\mathcal{K})\) the vertical tangent bundle of the fibre bundle \(\mathcal{G}/\mathcal{K}\).
	\end{enumerate} For any of these fibre bundles \(\mathcal{F}_1,\mathcal{F}_2\), denote by \(C^\infty_U(\mathcal{F}_1,\mathcal{F}_2)\) the smooth morphisms \(\mathcal{F}_1\to\mathcal{F}_2\) compatible with the projections to \(U\), which is also a fibre bundle over \(U\). 

The fibrewise WZW operator is then a smooth map \[\tilde{\tau}:C^\infty_U(\mathcal{D},\mathcal{G}/\mathcal{K})\to C^\infty_U(\mathcal{D},T(\mathcal{G}/\mathcal{K}))\times_U C^\infty_U(\d\mathcal{D},\mathcal{G}/\mathcal{K}).\] The fibrewise \(S^1\)-action on \((X|_U\times\C,p^*H)\) determined by the fibration degeneration yields a smooth section \(s:U\to C^\infty_U(\d \mathcal{D},\mathcal{G}/\mathcal{K})\), by pulling back the initial relatively cscK metric by the action. For each point \(b\in U\) there is then a unique solution \(f_b:D\to G_b/K_b\) of the WZW equation with boundary value \(s(b)\). Proposition \ref{thm:WZW_linearisation} together with the implicit function theorem for Banach manifolds implies the solutions \(f_b\) vary smoothly with \(b\).
\end{proof}

\subsection{Polystability with respect to product-type fibration degenerations}

\begin{theorem}\label{thm:N_limit}
	Let \((\X,\H)\to(B,L)\times\C\) be a product-type fibration degeneration of \((X,H)\to(B,L)\). Let \(\omega_X\) be a fixed choice of relatively cscK metric on \(X\), and let \(\varphi:[0,\infty)\to\K_E\) be the unique geodesic ray associated to \((\X,\H)\). Then \[\lim_{t\to\infty}\frac{\N(\varphi_t)}{t}=W_1(\X,\H).\]
\end{theorem}

To prove this, we will need the following lemma, which is a special case of a result by Boucksom--Hisamoto--Jonsson:

\begin{lemma}[{\cite[Lemma 3.9]{BHJ19}}]\label{lem:Deligne_limit}
	Let \((\X,\H)\) be a product-type fibration degeneration of \((X,H)\). Let \(\L_0,\ldots,\L_{m+n}\) be line bundles on \(\X\) with restrictions \(L_0,\ldots,L_{m+n}\) to \(X=\X_1\), and choose fixed hermitian metrics \(h_0,\ldots,h_{m+n}\) on \(L_0,\ldots,L_{m+n}\). We compactify \(\X\) and each of the \(\L_j\) trivially over \(\P^1\), and denote them by the \(\overline{\X}\) and \(\overline{\L}_j\). For each \(j\), let \(h_j\) be a smooth \(S^1\)-invariant metric on \(\L_j\) defined in a neighbourhood of \(0\in\C\), and write \(h_j^t\) for the corresponding ray of smooth metrics on \(L_j\). Then \[\langle h_0^t,\ldots,h_{m+n}^t\rangle_X-\langle h_0,\ldots,h_{m+n}\rangle_X=t\,(\overline{\L}_0\cdot\ldots\cdot\overline{\L}_{m+n})+\O(1).\]
\end{lemma}

\begin{proof}[Proof of Theorem \ref{thm:N_limit}]
	Let \(\varphi:[0,\infty)\to\K_E\) be the unique geodesic ray associated to the fibration degeneration; this corresponds to a smooth hermitian metric \(h_\X\) on \(\H|_D\) whose restriction to \(\H|_{\d D}\) is isomorphic to \(p^*h_X\). Denote by \(h_X^t\) the ray of smooth metrics on \(X\) determined by \(h_\X\). Then \[\varphi_t=\log(h_X^t/h_X)\] is a smooth geodesic ray in \(\K_E\). The metric \(h_\X\) also determines a smooth metric \(\eta_\X\) on the relative canonical bundle \(K_{\X/D}\), which is a line bundle in the case of product-type fibration degenerations. We then obtain a ray of smooth metrics \(\eta_X^t\) on \(K_X\). Proposition \ref{prop:Deligne_pairings} then gives \begin{align*}
	\widetilde{H}(\varphi_t)+\widetilde{R}(\varphi_t) &= \langle \eta_{X}^t,h_B^{n-1},(h_{X}^t)^{m+1}\rangle_{X}-\langle \eta_X,h_B^{n-1},h_{X}^{m+1}\rangle_{X}, \\
	\widetilde{I}(\varphi_t) &= \langle h_B^{n},(h_{X}^t)^{m+1}\rangle_{X}-\langle h_B^{n},h_{X}^{m+1}\rangle_{X}, \\
	\widetilde{J}(\varphi_t) &=
	\langle h_B^{n-1},(h_{X}^t)^{m+2}\rangle_{X}-\langle h_B^{n-1},h_{X}^{m+2}\rangle_{X}.
	\end{align*} Lemma \ref{lem:Deligne_limit} together with the Chen--Tian style formula Theorem \ref{prop:Chen-Tian_fib} and the intersection theory formula Proposition \ref{prop:W_1_int_theory} for \(W_1(\X,\H)\) then immediately imply the result. For example, applying Lemma \ref{lem:Deligne_limit} to the formula for \(\tilde{J}(\varphi_t)\) in terms of Deligne pairings, we get \[\lim_{t\to0}\frac{\widetilde{J}(\varphi_t)}{t}=L^{n-1}\cdot \H^{m+2}.\] From the Chen--Tian formula for \(\mathcal{N}\), we have \[\N(\varphi_t)=\binom{m+n}{n-1}\left(\widetilde{H}(\varphi_t)+\widetilde{R}(\varphi_t)+\frac{1}{n}A_1(X,H)\widetilde{I}(\varphi_t)+\frac{1}{m+2}A_0(X,H)\widetilde{J}(\varphi_t)\right),\] and from Proposition \ref{prop:W_1_int_theory}, \[W_1(\X,\H)=\binom{m+n}{n-1}(C_1(\X,\H)+C_2(\X,\H)+C_3(\X,\H)),\] where \begin{align*}
		C_1(\X,\H) &= \frac{1}{m+2}A_0(X,H) (L^{n-1}\cdot\overline{\H}^{\,m+2}), \\
		C_2(\X,\H) &= \frac{1}{n}A_1(X,H)(L^n\cdot\overline{\H}^{\,m+1}), \\
		C_3(\X,\H) &= (L^{n-1}\cdot\overline{\H}^{\,m+1}\cdot K_{\overline{\X}/\P^1}).
	\end{align*} It follows that \[\frac{1}{m+2}A_0(X,H)\frac{\widetilde{J}(\varphi_t)}{t}\to C_1(\X,\H)\] as \(t\to\infty\). Applying this same argument to \(\widetilde{H}(\varphi_t)+\widetilde{R}(\varphi_t)\) and \(\widetilde{I}(\varphi_t)\), and comparing terms in the Chen--Tian formula to terms in \(W_1(\X,\H)\), we get \[\lim_{t\to\infty}\frac{\mathcal{N}(\varphi_t)}{t}=W_1(\X,\H).\qedhere\]
\end{proof}

\begin{theorem}
	Let \((X,H)\to(B,L)\) be a fibration that admits an optimal symplectic connection in the class \(c_1(H)\), and let \((\X,\H)\) be a fibration degeneration of \((X,H)\) that is product-type. Then \(W_0(\X,\H)\geq0\), and if \(W_0(\X,\H)=0\) then \(W_1(\X,\H)\geq0\) with equality if and only if \((\X,\H)\) is a product fibration degeneration.
\end{theorem}

\begin{proof}
	We may assume that \(W_0(\X,\H)=0\), since the fibres \((X_b,H_b)\) are K-polystable. Noting that any optimal symplectic connection is a critical point of the functional \(\N\), the fact that \(W_1(\X,\H)\geq0\) follows from the previous theorem and convexity of the functional \(\N\) along geodesics.
	
	If \(W_1(\X,\H)=0\) then the geodesic ray \(\varphi_t\) must originate from a one-parameter subgroup of \(\Aut_0(\pi)\), by Theorem \ref{thm:fib_convexity}. The fibration degeneration is then isomorphic to the product fibration degeneration determined by the one-parameter subgroup. This can be seen by the result \cite[Lemma 3.11]{Sjo20}; note that in Sj{\"o}str{\"o}m Dyrefelt's terminology, our geodesic ray determines a subgeodesic ray on the test configuration \((\X,kL+\H)\) for all \(k\) sufficiently large. The result then says that there can be at most one normal test configuration compatible with our subgeodesic ray, and hence it must be the induced product test configuration.
\end{proof}

\begin{corollary}
	Let \((X,H)\to(B,L)\) be a fibration whose fibres \(X_b\) are all isomorphic as complex manifolds. If \((X,H)\to(B,L)\) admits an optimal symplectic connection in \(c_1(H)\), then \((X,H)\) is polystable with respect to fibration degenerations that are smooth over \(B\times\C\).
\end{corollary}

\begin{proof}
	This follows immediately from the previous theorem, together with Lemmas \ref{lem:a} and \ref{lem:b}.
\end{proof}

\begin{remark}\label{rem:Kahler_case}
	We have chosen here to work in the algebraic setting, however the results carry to the K{\"a}hler setting with minimal effort. Namely, in place of Lemma \ref{lem:Deligne_limit}, we instead use the result of Sj{\"o}str{\"o}m Dyrefelt \cite[Theorem 3.6]{Sjo18}, and work with K{\"a}hler forms instead of hermitian metrics.
\end{remark}

\subsection{Examples of product-type fibration degenerations}

We end by giving a large class of examples of product-type fibration degenerations. \begin{example}\label{ex:projectivised_vector_bundle}
	First, consider the case where \(X=\P(V)\) is a projectivised vector bundle over \(B\). If \(W\subset V\) is a sub-vector bundle, then there is a degeneration of vector bundles over \(B\) from \(V\) to \(W\oplus(V/W)\). That is, there is a sheaf \(\mathcal{V}\) on \(B\times\C\) that is flat over \(\C\), and whose general fibre over \(\C\) is the bundle \(V\) and whose central fibre is \(W\oplus(V/W)\). This degeneration of vector bundles naturally induces a fibration degeneration of \(X\) simply by taking \(\X:=\P(\mathcal{V})\) and \(\H:=\mathcal{O}_{\P(\mathcal{V})}(1)\). More generally, one could take a flag \(0=W_0\subset W_1\subset\cdots\subset W_k=V\) of vector subbundles, and this will similarly induce a product-type degeneration.
\end{example} 

\begin{example}[{McCarthy \cite[Chapter 7.2]{McC22}}]\label{ex:principal_bundle_fibration}
	Let \(Q\to B\) be a holomorphic principal \(G\)-bundle, where \(G\) is a reductive group, let \((Y,L_Y)\) be a polarised variety, and let \(\rho:G\to\Aut(Y,L_Y)\) be a representation. Letting \(X:=Q\times_\rho Y\) be the associated bundle with induced relatively ample line bundle \(H:=Q\times_\rho L_Y\), we have a fibration \((X,H)\to (B,L)\) all of whose fibres are isomorphic to \((Y,L_Y)\). 
	
	Recall a subgroup \(P\subset G\) is \emph{parabolic} if it contains a maximal solvable subgroup, called a \emph{Borel subgroup}. Parabolic subgroups can be described in the following way: a subgroup \(P\subset G\) is parabolic if and only if there exists a one-parameter subgroup \(\lambda:\C^*\to P\) such that \[P=\{g\in G:\lim_{t\to0}\lambda(t)g\lambda(t)^{-1}\text{ exists}\}.\] In concrete terms, we can always embed the reductive group \(G\) in some \(\GL(N,\C)\), such that \(P\) consists of matrices in \(G\) of the form \[
	\begin{pmatrix}
	A_1 & * & * & \cdots & * \\
	0 & A_2 & * & \cdots & * \\
	0 & 0 & A_3 & \cdots & * \\
	\vdots & \vdots & \vdots & \ddots & \vdots \\
	0 & 0 & 0 & \cdots & A_k
	\end{pmatrix}
	\] where the \(A_j\) are invertible square matrices of fixed size \(a_j\times a_j\), the \(*\)'s denote arbitrary matrices of the appropriate sizes, and \(\sum_{j=1}^ka_j=N\). The one-parameter subgroup \(\lambda\) can be taken to be of the form \[\lambda(t):=\begin{pmatrix}
	t^{r_1}I_{a_1} &0&0&\cdots&0\\
	0&t^{r_2}I_{a_2}&0&\cdots&0 \\
	0&0&t^{r_3}I_{a_3}&\cdots&0 \\
	\vdots&\vdots&\vdots&\ddots&\vdots \\
	0&0&0&\cdots&t^{r_k}I_{a_k}
	\end{pmatrix}\] where \(r_1>r_2>\cdots >r_k\) and \(I_{a_j}\) is the identity matrix of size \(a_j\). Writing \(L\) for the set of block-diagonal matrices with block sizes \((a_1,\ldots,a_k)\), and writing \(U\) for the set of strictly upper block triangular matrices in \(P\) with diagonal entries all \(1\), we have \(P=U\rtimes L\). The subset \(U\) is the \emph{unipotent radical} of \(P\), and is independent of the choice of \(\lambda\). The reductive subgroup \(L\) is called a \emph{Levi subgroup} of \(P\); it depends on the choice of \(\lambda\) but is unique up to conjugation.
	
	Given a parabolic subgroup \(P\subset G\), a \emph{reduction of the structure group} of \(Q\) to \(P\) is by definition a choice of open cover \(\{U_\alpha\}\) of \(B\) together with holomorphic transition functions \(g_{\alpha\beta}:U_{\alpha}\cap U_\beta\to P\) for the principal bundle \(Q\) which take values in \(P\). Such reductions notably arise in the definition of stability for principal bundles, given in \cite{AB01}. Given a reduction of structure group to \(P\), we can define a degeneration of the principal bundle as follows. On \(B\times\C\), take the open cover \(\{U_\alpha\times\C\}\), and define transition functions \(\tilde{g}_{\alpha\beta}:(U_\alpha\cap U_\beta)\times\C\to P\) by \[\tilde{g}_{\alpha\beta}(b,t):=\begin{cases}
	\lambda(t)g_{\alpha\beta}(b)\lambda(t)^{-1} & \text{ for } t\neq0, \\
	\lim_{t\to0}\lambda(t)g_{\alpha\beta}(b)\lambda(t)^{-1} & \text{ for }t=0.
	\end{cases}\] Let \(\mathcal{Q}\) be the principal \(P\)-bundle on \(B\times\C\) deterimined by these transition functions. The fibre of \(\mathcal{Q}\) over \(t\neq0\) isomorphic to the original principal bundle \(Q\), however the fibre over \(t=0\) is a principal \(P\)-bundle \(Q_0\) that may be different to \(Q\). Note that \(Q_0\) comes equipped with a reduction of the structure group from \(P\) to \(L\).  
	
	Taking the associated bundle \((\X,\H):=\mathcal{Q}\times_\rho(Y,L_Y)\), we obtain a fibration degeneration of \((X,H)\) that is of product-type. To write down the \(\C^*\)-action on this fibration degeneration, let us trivialise \(\mathcal{Q}\) over \(U_\alpha\times\C\), so that \(\mathcal{Q}|_{U_\alpha\times\C}\cong U_\alpha\times\C\times P\). An element of \((\X,\H)|_{U_\alpha\times\C}\) can then be written as \((b,t,[p,(y,\ell)])\), where \(b\in U_\alpha\), \(t\in\C\), and \([p,(y,\ell)]\) is an equivalence class in the Cartesian product \(P\times(Y,L_Y)\). The \(\C^*\)-action is defined by \[s\cdot(b,t,[p,(y,\ell)]):=(b,st,[\lambda(t)p\lambda(t)^{-1},\rho(\lambda(t))(y,\ell)]).\] It is straightforward to check that these actions respect the transition functions \(\tilde{g}_{\alpha\beta}\), and so glue to a well-defined \(\C^*\)-action.
\end{example}

We remark that any fibration whose fibres are all isomorphic arises from a holomorphic principal bundle; this is essentially a consequence of the Fischer--Grauert theorem \cite{FG65}, which states that a proper holomorphic fibration whose fibres are all isomorphic is locally trivial. For more information on fibrations associated to principal bundles, see the PhD thesis of J. McCarthy \cite{McC22}.

\bibliographystyle{plain}
\bibliography{polystability}

\begin{thebibliography}{10}

\bibitem{AB01}
B.~Anchouche and I.~Biswas.
\newblock Einstein-{H}ermitian connections on polystable principal bundles over
  a compact {K}\"{a}hler manifold.
\newblock {\em Amer. J. Math.}, 123(2):207--228, 2001.

\bibitem{BB17}
R.~J. Berman and B.~Berndtsson.
\newblock Convexity of the {$K$}-energy on the space of {K}\"{a}hler metrics
  and uniqueness of extremal metrics.
\newblock {\em J. Amer. Math. Soc.}, 30(4):1165--1196, 2017.

\bibitem{BDL16}
R.~J. Berman, T.~Darvas, and C.~H. Lu.
\newblock Regularity of weak minimizers of the {$K$}-energy and applications to
  properness and {$K$}-stability.
\newblock {\em Ann. Sci. \'{E}c. Norm. Sup\'{e}r. (4)}, 53(2):267--289, 2020.

\bibitem{Blo13}
Z.~B{\l}ocki.
\newblock The complex monge--ampere equation in k{\"a}hler geometry.
\newblock In {\em Pluripotential theory}, pages 95--141. Springer, 2013.

\bibitem{BHJ19}
S.~Boucksom, T.~Hisamoto, and M.~Jonsson.
\newblock Uniform {K}-stability and asymptotics of energy functionals in
  {K}\"{a}hler geometry.
\newblock {\em J. Eur. Math. Soc. (JEMS)}, 21(9):2905--2944, 2019.

\bibitem{Bro11}
T.~Br{\"o}nnle.
\newblock {\em Deformation constructions of extremal metrics}.
\newblock PhD thesis, Imperial College, University of London, 2011.

\bibitem{Che19}
G.~Chen.
\newblock The {J}-equation and the supercritical deformed
  {H}ermitian-{Y}ang-{M}ills equation.
\newblock {\em Invent. Math.}, 225(2):529--602, 2021.

\bibitem{Che00b}
X.~Chen.
\newblock On the lower bound of the {M}abuchi energy and its application.
\newblock {\em Internat. Math. Res. Notices}, (12):607--623, 2000.

\bibitem{Che00}
X.~Chen.
\newblock The space of {K}\"{a}hler metrics.
\newblock {\em J. Differential Geom.}, 56(2):189--234, 2000.

\bibitem{CTW18}
J.~Chu, V.~Tosatti, and B.~Weinkove.
\newblock {$C^{1,1}$} regularity for degenerate complex {M}onge-{A}mp\`ere
  equations and geodesic rays.
\newblock {\em Comm. Partial Differential Equations}, 43(2):292--312, 2018.

\bibitem{CY18}
T.~C. Collins and S.-T. Yau.
\newblock Moment maps, nonlinear {PDE} and stability in mirror symmetry, {I}:
  geodesics.
\newblock {\em Ann. PDE}, 7(1):Paper No. 11, 73, 2021.

\bibitem{DaR17}
T.~Darvas and Y.~A. Rubinstein.
\newblock Tian's properness conjectures and {F}insler geometry of the space of
  {K}\"{a}hler metrics.
\newblock {\em J. Amer. Math. Soc.}, 30(2):347--387, 2017.

\bibitem{DR17}
R.~Dervan and J.~Ross.
\newblock K-stability for {K}\"{a}hler manifolds.
\newblock {\em Math. Res. Lett.}, 24(3):689--739, 2017.

\bibitem{DS_moduli}
R.~Dervan and L.~M. Sektnan.
\newblock Moduli theory, stability of fibrations and optimal symplectic
  connections.
\newblock {\em Geom. Topol.}, 25(5):2643--2697, 2021.

\bibitem{DS_uniqueness}
R.~Dervan and L.~M. {Sektnan}.
\newblock Uniqueness of optimal symplectic connections.
\newblock {\em Forum Math. Sigma}, 9(e18):1--37, 2021.

\bibitem{DS_osc}
Ruadha\'{\i} Dervan and Lars~Martin Sektnan.
\newblock Optimal symplectic connections on holomorphic submersions.
\newblock {\em Comm. Pure Appl. Math.}, 74(10):2132--2184, 2021.

\bibitem{Don85}
S.~K. Donaldson.
\newblock Anti self-dual {Y}ang-{M}ills connections over complex algebraic
  surfaces and stable vector bundles.
\newblock {\em Proc. London Math. Soc. (3)}, 50(1):1--26, 1985.

\bibitem{Don92}
S.~K. Donaldson.
\newblock Boundary value problems for {Y}ang-{M}ills fields.
\newblock {\em J. Geom. Phys.}, 8(1-4):89--122, 1992.

\bibitem{Don99}
S.~K. Donaldson.
\newblock Symmetric spaces, {K}\"{a}hler geometry and {H}amiltonian dynamics.
\newblock In {\em Northern {C}alifornia {S}ymplectic {G}eometry {S}eminar},
  volume 196 of {\em Amer. Math. Soc. Transl. Ser. 2}, pages 13--33. Amer.
  Math. Soc., Providence, RI, 1999.

\bibitem{Don02}
S.~K. Donaldson.
\newblock Scalar curvature and stability of toric varieties.
\newblock {\em J. Differential Geom.}, 62(2):289--349, 2002.

\bibitem{Don05}
S.~K. Donaldson.
\newblock Lower bounds on the {C}alabi functional.
\newblock {\em J. Differential Geom.}, 70(3):453--472, 2005.

\bibitem{FG65}
W.~Fischer and H.~Grauert.
\newblock Lokal-triviale {F}amilien kompakter komplexer {M}annigfaltigkeiten.
\newblock {\em Nachr. Akad. Wiss. G\"{o}ttingen Math.-Phys. Kl. II},
  1965:89--94, 1965.

\bibitem{Gau10}
P.~Gauduchon.
\newblock Calabi's extremal metrics: an elementary introduction.
\newblock Manuscript, 2010.

\bibitem{GRS21}
V.~Georgoulas, J.~W. Robbin, and D.~A. Salamon.
\newblock {\em The moment-weight inequality and the {H}ilbert-{M}umford
  criterion---{GIT} from the differential geometric viewpoint}, volume 2297 of
  {\em Lecture Notes in Mathematics}.
\newblock Springer, Cham, [2021] \copyright 2021.

\bibitem{Gol72}
R.~A. Goldstein.
\newblock Stability of the boundary-value problem for harmonic mappings.
\newblock {\em J. Math. Anal. Appl.}, 39:346--359, 1972.

\bibitem{Hal22}
M.~Hallam.
\newblock {\em The geometry and stability of fibrations}.
\newblock PhD thesis, Oxford University, 2022.

\bibitem{HK18}
Y.~Hashimoto and J.~Keller.
\newblock Quot-scheme limit of {F}ubini-study metrics and {D}onaldson's
  functional for vector bundles.
\newblock {\em \'{E}pijournal G\'{e}om. Alg\'{e}brique}, 5:Art. 21, 38, 2021.

\bibitem{LS94}
C.~LeBrun and S.~R. Simanca.
\newblock Extremal {K}\"{a}hler metrics and complex deformation theory.
\newblock {\em Geom. Funct. Anal.}, 4(3):298--336, 1994.

\bibitem{LS15}
M.~Lejmi and G.~Sz\'{e}kelyhidi.
\newblock The {J}-flow and stability.
\newblock {\em Adv. Math.}, 274:404--431, 2015.

\bibitem{Mab86}
T.~Mabuchi.
\newblock {$K$}-energy maps integrating {F}utaki invariants.
\newblock {\em Tohoku Math. J. (2)}, 38(4):575--593, 1986.

\bibitem{Mab87}
T.~Mabuchi.
\newblock Some symplectic geometry on compact {K}\"{a}hler manifolds. {I}.
\newblock {\em Osaka J. Math.}, 24(2):227--252, 1987.

\bibitem{McC22}
J.~B. McCarthy.
\newblock {\em Stability conditions and canonical metrics}.
\newblock PhD thesis, Imperial College London, 2022.

\bibitem{New78}
P.~E. Newstead.
\newblock {\em Introduction to moduli problems and orbit spaces}, volume~51 of
  {\em Tata Institute of Fundamental Research Lectures on Mathematics and
  Physics}.
\newblock Tata Institute of Fundamental Research, Bombay; by the Narosa
  Publishing House, New Delhi, 1978.

\bibitem{Oda13}
Y.~Odaka.
\newblock A generalization of the {R}oss-{T}homas slope theory.
\newblock {\em Osaka J. Math.}, 50(1):171--185, 2013.

\bibitem{PRS08}
D.~H. Phong, J.~Ross, and J.~Sturm.
\newblock Deligne pairings and the {K}nudsen-{M}umford expansion.
\newblock {\em J. Differential Geom.}, 78(3):475--496, 2008.

\bibitem{PS10}
D.~H. Phong and J.~Sturm.
\newblock The {D}irichlet problem for degenerate complex {M}onge-{A}mpere
  equations.
\newblock {\em Comm. Anal. Geom.}, 18(1):145--170, 2010.

\bibitem{RT07}
J.~Ross and R.~Thomas.
\newblock A study of the {H}ilbert-{M}umford criterion for the stability of
  projective varieties.
\newblock {\em J. Algebraic Geom.}, 16(2):201--255, 2007.

\bibitem{Sem92}
S.~Semmes.
\newblock Complex {M}onge-{A}mp\`ere and symplectic manifolds.
\newblock {\em Amer. J. Math.}, 114(3):495--550, 1992.

\bibitem{Sjo18}
Z.~Sj{\"{o}}str{\"{o}}m~Dyrefelt.
\newblock K-semistability of csc{K} manifolds with transcendental cohomology
  class.
\newblock {\em J. Geom. Anal.}, 28(4):2927--2960, 2018.

\bibitem{Sjo20}
Z.~Sj{\"{o}}str{\"{o}}m~Dyrefelt.
\newblock On {K}-polystability of csc{K} manifolds with transcendental
  cohomology class.
\newblock {\em Int. Math. Res. Not. IMRN}, (9):2769--2817, 2020.

\bibitem{Sze10}
G.~Sz{\'{e}}kelyhidi.
\newblock The {K}\"{a}hler-{R}icci flow and {$K$}-polystability.
\newblock {\em Amer. J. Math.}, 132(4):1077--1090, 2010.

\bibitem{Sze14}
G.~Sz{\'{e}}kelyhidi.
\newblock {\em An introduction to extremal {K}\"{a}hler metrics}, volume 152 of
  {\em Graduate Studies in Mathematics}.
\newblock American Mathematical Society, Providence, RI, 2014.

\bibitem{Tia97}
G.~Tian.
\newblock K\"{a}hler-{E}instein metrics with positive scalar curvature.
\newblock {\em Invent. Math.}, 130(1):1--37, 1997.

\bibitem{Tia00}
G.~Tian.
\newblock {\em Canonical metrics in {K}\"{a}hler geometry}.
\newblock Lectures in Mathematics ETH Z\"{u}rich. Birkh\"{a}user Verlag, Basel,
  2000.
\newblock Notes taken by Meike Akveld.

\bibitem{UY86}
K.~Uhlenbeck and S.-T. Yau.
\newblock On the existence of {H}ermitian-{Y}ang-{M}ills connections in stable
  vector bundles.
\newblock volume~39, pages S257--S293. 1986.
\newblock Frontiers of the mathematical sciences: 1985 (New York, 1985).

\bibitem{Wan12}
X.~Wang.
\newblock Height and {GIT} weight.
\newblock {\em Math. Res. Lett.}, 19(4):909--926, 2012.

\bibitem{Yau93}
S.-T. Yau.
\newblock Open problems in geometry.
\newblock In {\em Differential geometry: partial differential equations on
  manifolds ({L}os {A}ngeles, {CA}, 1990)}, volume~54 of {\em Proc. Sympos.
  Pure Math.}, pages 1--28. Amer. Math. Soc., Providence, RI, 1993.

\end{thebibliography}

\end{document}